\documentclass{amsart}

\usepackage{amsmath}
\usepackage{amssymb}
\usepackage{amsthm}
\usepackage{graphicx}
\usepackage{enumitem}

\theoremstyle{plain}
\newtheorem*{thm*}{Theorem}
\newtheorem{thm}{Theorem}
\newtheorem{thml}{Theorem}
\newtheorem{prop}[thm]{Proposition}
\newtheorem{lem}[thm]{Lemma}

\theoremstyle{definition}
\newtheorem{defn}[thm]{Definition}
\newtheorem{claim}{Claim}[thm]
\newtheorem{step}{Step}
\newtheorem{question}{Question}

\newcommand{\intr}{\mathrm{int}}
\newcommand{\diam}{\mathrm{diam}}
\newcommand{\dsup}{d_{\mathrm{sup}}}
\newcommand{\id}{\mathrm{id}}
\newcommand{\setends}[1]{\mathcal{E}(#1)}

\begin{document}

\title[Homogeneous plane continua]{A complete classification of\\homogeneous plane continua}
\author{L. C. Hoehn \and L. G. Oversteegen}
\date{\today}

\address[L.\ C.\ Hoehn]{Nipissing University, Department of Computer Science \& Mathematics, 100 College Drive, Box 5002, North Bay, Ontario, Canada, P1B 8L7}
\email{loganh@nipissingu.ca}

\address[L.\ G.\ Oversteegen]{University of Alabama at Birmingham, Department of Mathematics, Birmingham, AL 35294, USA}
\email{overstee@uab.edu}

\dedicatory{Dedicated to Andrew Lelek on the occasion of his 80th birthday}

\thanks{The first named author was partially supported by NSERC grant RGPIN 435518, and by the Mary Ellen Rudin Young Researcher Award}


\begin{abstract}
We show that the only compact and connected subsets (i.e.\ \emph{continua}) $X$ of the plane $\mathbb{R}^2$ which contain more than one point and are homogeneous, in the sense that the group of homeomorphisms of $X$ acts transitively on $X$, are, up to homeomorphism, the circle $\mathbb{S}^1$, the pseudo-arc, and the circle of pseudo-arcs.  These latter two spaces are fractal-like objects which do not contain any arcs.  It follows that any compact and homogeneous space in the plane has the form $X \times Z$, where $X$ is either a point or one of the three homogeneous continua above, and $Z$ is either a finite set or the Cantor set.

The main technical result in this paper is a new characterization of the pseudo-arc.  Following Lelek, we say that a continuum $X$ has \emph{span zero} provided for every continuum $C$ and every pair of maps $f,g: C \to X$ such that $f(C) \subset g(C)$ there exists $c_0 \in C$ so that $f(c_0) = g(c_0)$.  We show that a continuum is homeomorphic to the pseudo-arc if and only if it is hereditarily indecomposable (i.e., every subcontinuum is indecomposable) and has span zero.
\end{abstract}

\maketitle

\section{Introduction}
\label{sec:intro}

By a \emph{compactum}, we mean a compact metric space, and by a \emph{continuum}, we mean a compact connected metric space.  A continuum is \emph{non-degenerate} if it contains more than one point.  We refer to the space $\mathbb{R}^2$, with the Euclidean topology, as \emph{the plane}.  By a \emph{map} we mean a continuous function.

A space $X$ is (topologically) \emph{homogeneous} if for every $x,y \in X$ there exists a homeomorphism $h: X \to X$ with $h(x) = y$.  All homeomorphisms in this paper are onto.

The concept of topological homogeneity was first introduced by Sierpi{\'n}ski in \cite{sierpinski20}.  Since the underlying/ambient space of many topological models is homogeneous, the classification of homogeneous spaces has a long and rich history.  For example, all connected manifolds are homogeneous, and the Hilbert cube $[0,1]^{\mathbb{N}}$,  which contains a homeomorphic copy of every compact metric space, is an example of an infinite dimensional homogeneous continuum.  Even for low dimensions, the classification of homogeneous Riemannian manifolds remains an active area of research today.  Contrary to naive expectation, homogeneous continua do not necessarily have a simple local structure (in particular, they do not need to contain a manifold).  As a consequence, even the classification of one-dimensional homogeneous continua appears out of reach.  This paper concerns the classification of homogeneous compact subsets of the plane.

In the first volume of Fundamenta Mathematicae in 1920, Knaster and Kuratowski \cite{KK20} asked (Probl\`{e}me 2) whether the circle is the only (non-degenerate) homogeneous plane continuum.  Mazurkiewicz \cite{mazurkiewicz24} showed early on that the answer is yes if the continuum is locally connected.  Cohen \cite{cohen51} showed that the answer is yes if the continuum is arcwise connected or, equivalently, pathwise connected, and Bing \cite{bing60} proved more generally that the answer remains yes if the continuum  simply contains an arc.  A continuum $X$ is \emph{decomposable} if it is the union of two proper subcontinua and \emph{indecomposable} otherwise.  A continuum is \emph{hereditarily decomposable (hereditarily indecomposable)} if every non-degenerate subcontinuum is decomposable (indecomposable, respectively).  Hagopian \cite{hagopian75} showed that the answer to the question of Knaster and Kuratowski is still yes if the continuum merely contains a hereditarily decomposable subcontinuum.

Probl\`{e}me 2 by Knaster and Kuratowski was formally solved by Bing \cite{bing48} who showed in 1948 that the pseudo-arc, described in detail in Section \ref{sec:pseudo-arc}, is another homogeneous plane continuum.  The pseudo-arc is a one-dimensional fractal-like hereditarily indecomposable continuum (in particular it contains no arcs).  This stunning example of a homogeneous continuum shows that homogeneity is possible at two extremes: one where the local structure is simple (e.g.\ for locally connected spaces) and one where the local structure is not simple (e.g.\ for not locally connected spaces).  Since Bing's surprising solution, the question has been: What are all homogeneous plane continua?  A third homogeneous plane continuum, called the circle of pseudo-arcs (since it admits an open map to the circle whose point pre-images are all pseudo-arcs), was added by Bing and Jones \cite{BJ59} in 1954.  We show in this paper that these three comprise the complete list of all homogeneous non-degenerate plane continua.

Even though hereditarily indecomposable continua seem to be obscure objects, they arise naturally in mathematics, for example as attractors in dynamical systems \cite{KY94} (even for an open set of parameters).

Another hereditarily indecomposable continuum, the \emph{pseudo-circle}, was considered to be a strong candidate to be an additional example of a homogeneous plane continuum.  However, it was proved to be not homogeneous independently by Fearnley \cite{fearnley69} and Rogers \cite{rogers70}.

This long-standing question of the classification of all homogeneous plane continua has been raised and/or addressed in several papers and surveys, including: \cite{jones69}, \cite{jones73}, \cite{jones75}, \cite{lewis80}, \cite{lewis83}, \cite{lewis84}, \cite{rogers83}, \cite{rogers92}, \cite{rudin97}, and the ``New Scottish Book'' (Problem 920).  The first explicit statement concerning this problem that we could find is in \cite{jones51}.

There exists a rich literature concerning homogeneous continua (including several excellent surveys, such as \cite{lewis92}, \cite{rogers83}, and \cite{rogers92}) so we will only briefly state some pertinent highlights here.

In 1954, Jones proved the following result:

{\renewcommand*{\thethml}{\Alph{thml}}
\begin{thml}[\cite{jones55}]
\label{thml:homog decomp circle}
If $M$ is a decomposable homogeneous continuum in the plane, then $M$ is a circle of mutually homeomorphic indecomposable homogeneous continua.
\end{thml}
}

The conclusion of this theorem implies that there is an indecomposable homogeneous continuum $X$ (possibly a single point) and an open map from $M$ to the circle all of whose point preimages are homeomorphic to $X$.  Bing and Jones \cite{BJ59} constructed in 1954 such a continuum in the plane for which $X$ is the pseudo-arc, and proved it is homogeneous.  This example is known as the \emph{circle of pseudo-arcs} (see Section \ref{sec:pseudo-arc}).

It follows from this theorem of Jones that every decomposable homogeneous continuum in the plane separates the plane.  Rogers \cite{rogers81} proved that conversely, every homogeneous plane continuum which separates the plane is decomposable.

Hagopian (see also \cite{jones69}) obtained in 1976 the following result:

{\renewcommand*{\thethml}{\Alph{thml}}
\begin{thml}[\cite{hagopian76}]
\label{thml:homog indec hered indec}
Every indecomposable homogeneous plane continuum is hereditarily indecomposable.
\end{thml}
}

A map $f: X \to Y$ is called an \emph{$\varepsilon$-map} if for each $y \in Y$, $\diam(f^{-1}(y)) < \varepsilon$.  A continuum $X$ is \emph{arc-like} (respectively, \emph{tree-like}) provided for each $\varepsilon > 0$ there exists an $\varepsilon$-map from $X$ to an arc (respectively, tree).  Bing \cite{bing51} proved in 1951 that the pseudo-arc is the only hereditarily indecomposable arc-like continuum.  Hence, to show that an indecomposable homogeneous plane continuum is homeomorphic to the pseudo-arc, by the results of Hagopian and Bing it suffices to show that it is arc-like.

The main idea of our proof is based on a generalization of the following simple fact, which is central to much work done with the pseudo-arc.

\medskip\noindent
\textbullet\ Let $f: [0,1] \to [0,1]$ be a piecewise linear map.  For any $\varepsilon > 0$, if $g: [0,1] \to [0,1]$ is a sufficiently crooked map, then there is a map $h: [0,1] \to [0,1]$ such that $f \circ h$ is $\varepsilon$-close to $g$.
\medskip

\noindent
See Section \ref{sec:pseudo-arc} for a formal definition of ``crookedness''.  See also Theorem \ref{thm:factor fold} and Theorem \ref{thm:lift} below for related properties.

We will prove a generalization of the above statement, where instead of $[0,1]$ we consider graphs, and we restrict to a certain class of piecewise linear maps $f$.  To describe how this result pertains to the study of homogeneous plane continua, we provide some context below.

It is in general a difficult task to prove that a given continuum is (or is not) arc-like.  A closely related notion, introduced by Lelek in 1964 \cite{lelek64}, is that of span zero.  A continuum $X$ has \emph{span zero} if for any continuum $C$ and any two maps $f,g: C \to X$ such that $f(C) \subseteq g(C)$, there exist $p \in C$ with $f(p) = g(p)$ (by \cite{davis84} this is equivalent to the traditional definition of span zero where the images of $f$ and $g$ coincide).  It is easy to see that every arc-like continuum has span zero \cite{lelek64}.  Moreover, in some cases it is easier to show that a continuum $X$ has span zero than to show that it is arc-like.  For example, the following theorem was obtained in the early 1980's:

{\renewcommand*{\thethml}{\Alph{thml}}
\begin{thml}[\cite{OT82}]
\label{thml:homog span zero}
Every homogeneous indecomposable plane continuum has span zero.
\end{thml}
}

It was a long standing open problem whether each continuum of span zero is arc-like.  Unfortunately the answer was shown to be negative in \cite{hoehn11}.  The example given in \cite{hoehn11} relied heavily on the existence of patterns which required the continuum to contain arcs.  Such patterns are not possible for hereditarily indecomposable continua.  Indeed, using our generalization of the above result about crooked maps between arcs, we prove the following in this paper:

\begin{thm}
\label{thm:hered indec span zero}
A non-degenerate continuum $X$ is homeomorphic to the pseudo-arc if and only if $X$ is hereditarily indecomposable and has span zero.
\end{thm}

We suspect that this result will be useful in other contexts as well, for example, in the classification of attractors in certain dynamical systems.

It follows immediately from Theorems \ref{thml:homog indec hered indec} and \ref{thml:homog span zero} above and our Theorem \ref{thm:hered indec span zero} that every indecomposable non-degenerate homogeneous plane continuum is a pseudo-arc.  Combining this with Theorem \ref{thml:homog decomp circle} above, we obtain the following classification of homogeneous plane continua:

\begin{thm}
\label{thm:homog plane cont}
Up to homeomorphism, the only non-degenerate homogeneous continua in the plane are:
\begin{enumerate}
\item The circle;
\item The pseudo-arc; and
\item The circle of pseudo-arcs.
\end{enumerate}
\end{thm}

Finally, if $Y$ is a homogeneous compactum then by \cite{MR89} (see also \cite{AO90} and \cite{MR90}) $Y$ is homeomorphic to $X \times Z$, where $X$ is a homogeneous continuum and $Z$ is a $0$-dimensional homogeneous compactum and, hence, either a finite set or the Cantor set.  Thus we obtain the following corollary:

\begin{thm}
\label{thm:homog plane comp}
Up to homeomorphism, the only homogeneous compact spaces in the plane are:
\begin{enumerate}
\item Finite sets;
\item The Cantor set; and
\item The spaces $X \times Z$, where $X$ is a circle, pseudo-arc, or circle of pseudo-arcs, and $Z$ is either a finite set or the Cantor set.
\end{enumerate}
\end{thm}

\bigskip
This paper is organized as follows.  After fixing some definitions and notation in Section \ref{sec:notation}, we draw a connection in Section \ref{sec:span sep} between the property of span zero and sets in the product of a graph $G$ and the interval $[0,1]$ which separate $G \times \{0\}$ from $G \times \{1\}$.  For the remainder of the paper after this, we focus our attention on these separators, rather than work with span directly.  In Section \ref{sec:simple folds}, we characterize hereditarily indecomposable compacta in terms of simple piecewise linear functions between graphs.

In Section \ref{sec:stairwells}, we introduce a special type of separating set in the product of a graph with the interval, and prove that such separators are in a certain sense dense in the set of all separators.  Section \ref{sec:unfolding} is devoted to some technical results towards showing that such special separators can be ``unfolded'' by simple piecewise linear maps.  Finally, in Section \ref{sec:applications} we bring everything together and prove our main result, Theorem \ref{thm:hered indec span zero} above.  Section \ref{sec:discussion} includes some discussion and open questions.

\subsection{The pseudo-arc}
\label{sec:pseudo-arc}

In this subsection we give a brief introduction to the pseudo-arc, and describe some of its most important properties.

The pseudo-arc is the most well-known example of a hereditarily indecomposable continuum.  It is a very exotic and complex space with many remarkable and strange properties, yet it is also in some senses ubiquitous and quite natural.

Most descriptions of the pseudo-arc involve some notion of ``crookedness''.  We will appeal to the notion of a crooked map, as follows.

An onto map $g: [0,1] \to [0,1]$ is considered crooked if, roughly speaking, as $x$ travels from $0$ to $1$, $g(x)$ goes back and forth many times, on large and on small scales in $[0,1]$.  More precisely, given $\delta > 0$, we say $g$ is \emph{$\delta$-crooked} if there is a finite set $F \subset [0,1]$ which is a $\delta$-net for $[0,1]$ (i.e.\ each point of $[0,1]$ is within distance $\delta$ from some point of $F$), such that whenever $y_1,y_2,y_3,y_4$ is an increasing or decreasing sequence of points in $F$, and $x_1,x_4 \in [0,1]$ with $x_1 < x_4$ and $g(x_1) = y_1$, $g(x_4) = y_4$, there are points $x_2,x_3 \in [0,1]$ such that $x_1 < x_2 < x_3 < x_4$ and $g(x_2) = y_3$, $g(x_3) = y_2$.

To construct the pseudo-arc, one should choose a sequence of onto maps $g_n: [0,1] \to [0,1]$, $n = 1,2,\ldots$, such that for each $n$ and each $1 \leq k \leq n$, the composition $g_k \circ g_{k+1} \circ \cdots \circ g_n$ is $\frac{1}{n}$-crooked.  The \emph{pseudo-arc} is then the inverse limit of this sequence, $\varprojlim ([0,1], g_n)$.

The pseudo-arc, as constructed by this procedure, is a hereditarily indecomposable arc-like continuum.  According to Bing's characterization theorem \cite{bing51}, any two continua which are both hereditarily indecomposable and arc-like are homeomorphic.  Thus the pseudo-arc is the unique continuum with these properties.  This also means that the particular choices of maps $g_n$ in the above construction are not important -- so long as the crookedness properties are satisfied, the resulting inverse limit will be the same space.

One can equivalently construct the pseudo-arc in the plane as the intersection of a nested sequence of ``snakes'' (homeomorphs of the closed unit disk) which are nested inside one another in a manner reminiscent of the crooked pattern for maps described above.

Because of the enormous extent of crookedness inherent in the pseudo-arc, it is impossible to draw an informative, accurate raster image of this space (see \cite{LM10} for a detailed explanation).  Nevertheless, the pseudo-arc is in some sense ubiquitous: in any manifold $M$ of dimension at least $2$, the set of subcontinua homeomorphic to the pseudo-arc is a dense $G_\delta$ subset of the set of all subcontinua of $M$ (equipped with the Vietoris topology).  The pseudo-arc is a universal object in the sense that it is arc-like, and every arc-like continua is a continuous image of it.

The pseudo-arc has an interesting history of discovery.  It was first constructed by Knaster \cite{knaster22} in 1922 as the first example of a hereditarily indecomposable continuum.  Moise \cite{moise48} in 1948 constructed a similar example, which has the remarkable property that it is homeomorphic to each of its non-degenerate subcontinua.  Moise named this space the ``pseudo-arc'', since the interval $[0,1] \subset \mathbb{R}$ is the only other known space which shares this same property.  Also in 1948, Bing \cite{bing48} constructed another similar example which he proved was homogeneous, thus answering the original question of Knaster and Kuratowski about homogeneous continua in the plane.  Shortly after this, in 1951 Bing published the characterization theorem stated above, from which it follows that all three of these examples are in fact the same space.

Not only is the pseudo-arc homogeneous, but in fact it satisfies the following stronger properties: 1) given a collection of $n$ points $x_1,\ldots,x_n$, no two of which belong to any proper subcontinuum of the pseudo-arc, and given another such collection $y_1,\ldots,y_n$, there is a homeomorphism $h$ of the pseudo-arc to itself such that $h(x_i) = y_i$ for each $i = 1,\ldots,n$ \cite{lehner61}; and 2) given two points $x$ and $y$ and an open subset $U$, if there is a subcontinuum of the pseudo-arc containing $x$ and $y$ which is disjoint from $\overline{U}$, then there is a homeomorphism $h$ of the pseudo-arc to itself such that $h(x) = y$ and $h$ is the identity on $U$ \cite{lewis69}.  These properties should be compared with similar properties enjoyed by the circle $\mathbb{S}^1$: 1$'$) given two sets of $n$ points $x_1,\ldots,x_n,y_1,\ldots,y_n \in \mathbb{S}^1$, both arranged in circular order, there is a homeomorphism $h$ of $\mathbb{S}^1$ to itself such that $h(x_i) = y_i$ for each $i = 1,\ldots,n$; and 2$'$) given two points $x$ and $y$ and an open subset $U$, if there is a subarc of $\mathbb{S}^1$ containing $x$ and $y$ which is disjoint from $\overline{U}$, then there is a homeomorphism $h$ of $\mathbb{S}^1$ to itself such that $h(x) = y$ and $h$ is the identity on $U$.


In 1954, Bing and Jones \cite{BJ59} constructed a space called the \emph{circle of pseudo-arcs}.  This is a circle-like continuum which admits an open map to the circle whose point pre-images are pseudo-arcs (a continuum $X$ is \emph{circle-like} if for any $\varepsilon > 0$ there exists an $\varepsilon$-map from $X$ to the circle $\mathbb{S}^1$).  Bing and Jones proved that the circle of pseudo-arcs is homogeneous, and that it is unique, in the sense that it is the only continuum (up to homoemorphism) with the above properties.  The circle of pseudo-arcs should not be confused with the product of the pseudo arc with $\mathbb{S}^1$ (which is homogeneous but not embeddable in the plane), or with another related space called the \emph{pseudo-circle} (which is a hereditarily indecomposable circle-like continuum in the plane, but is not homogeneous -- see \cite{fearnley69} and \cite{rogers70}).

\section{Definitions and notation}
\label{sec:notation}

An \emph{arc} is a space homeomorphic to the interval $[0,1]$.  A \emph{graph} is a space which is the union of finitely many arcs which intersect at most in endpoints.  Given a graph $G$ and a point $x \in G$, $x$ is an \emph{endpoint} if $x$ is not a cutpoint of any connected neighborhood of $x$ in $G$, and $x$ is a \emph{branch point} if $x$ is a cutpoint of order $\geq 3$ in some connected neighborhood of $x$ in $G$.

The \emph{Hilbert cube} is the space $[0,1]^{\mathbb{N}}$, with the standard product metric $d$.  It has the property that any compact metric space embeds in it.  For this reason, we will assume throughout this paper that any compacta we consider are embedded in $[0,1]^{\mathbb{N}}$, and use this same metric $d$ for all of them.

Given two functions $f,g: X \to Y$ between compacta $X$ and $Y$, we use the \emph{supremum metric} to measure the distance between $f$ and $g$, defined by
\[ \dsup(f,g) = \sup \{d(f(x),g(x)): x \in X\} .\]

Given two non-empty subsets $A$ and $B$ of a compactum $X$, the \emph{Hausdorff distance} between $A$ and $B$ is
\[ d_H(A,B) = \inf \{\varepsilon > 0: A \subset B_\varepsilon \textrm{ and } B \subset A_\varepsilon\} ,\]
where $A_\varepsilon$ (respectively, $B_\varepsilon$) is the $\varepsilon$-neighborhood of $A$ (respectively, $B$).  It is well known that the hyperspace of all non-empty compact subsets of $X$, equipped with the Hausdorff metric, is compact.

\section{Span and separators}
\label{sec:span sep}

In this section, we draw a correspondance between the property of span zero and the existence of certain separating sets in the product of a graph and an arc which approximate a continuum.

As in the Introduction, a continuum $X$ has \emph{span zero} if whenever $f,g: C \to X$ are maps of a continuum $C$ to $X$ with $f(C) \subseteq g(C)$, there is a point $p \in C$ such that $f(p) = g(p)$.  This can equivalently be formulated as follows: $X$ has span zero if every subcontinuum $Z \subseteq X \times X$ with $\pi_1(Z) \subseteq \pi_2(Z)$ meets the diagonal $\Delta X = \{(x,x): x \in X\}$ (here $\pi_1$ and $\pi_2$ are the first and second coordinate projections $X \times X \to X$, respectively).  By \cite{davis84}, this is equivalent to the traditional definition of span zero where one insists that $\pi_1(Z) = \pi_2(Z)$.

The proof of the following theorem is implicit in results of \cite{OT84}.  We include a self-contained proof here for completeness.

We remark that in fact the property of ``span zero'' in this theorem could be replaced with the weaker property of ``surjective semispan zero'', which has the same definition as span zero except that one insists that $\pi_1(Z) \subseteq \pi_2(Z) = X$ \cite{lelek77}.

\begin{thm}
\label{thm:span separator}
Let $X \subset [0,1]^{\mathbb{N}}$ be a continuum in the Hilbert cube with span zero.  For any $\varepsilon > 0$, there exists $\delta > 0$ such that if $G \subset [0,1]^{\mathbb{N}}$ is a graph and $I \subset [0,1]^{\mathbb{N}}$ is an arc with endpoints $p$ and $q$, such that the Hausdorff distance from $X$ to each of $G$ and $I$ is less than $\delta$, then the set $M = \{(x,y) \in G \times (I \smallsetminus \{p,q\}): d(x,y) < \varepsilon\}$ separates $G \times \{p\}$ from $G \times \{q\}$ in $G \times I$.
\end{thm}

\begin{proof}
If the Theorem were false, then there would exist $\varepsilon > 0$ and a sequence of graphs $\langle G_n \rangle_{n=1}^\infty$ and arcs $\langle I_n \rangle_{n=1}^\infty$ with endpoints $p_n$ and $q_n$ in $[0,1]^{\mathbb{N}}$, both converging to $X$ in the Hausdorff metric, and such that the set $M_n = \{(x,y) \in G_n \times (I_n \smallsetminus \{p_n,q_n\}): d(x,y) < \varepsilon\}$ does not separate $G_n \times \{p_n\}$ from $G_n \times \{q_n\}$ for each $n = 1,2,\ldots$.  This would mean (see e.g.\ \cite[Theorem 5.2]{nadler92}) that for every $n = 1,2,\ldots$, there is a continuum $Z_n \subset G_n \times I_n$ meeting $G_n \times \{p_n\}$ and $G_n \times \{q_n\}$ (hence the second coordinate projection of $Z_n$ is all of $I_n$), such that $d(x,y) \geq \varepsilon$ for all $(x,y) \in Z_n$.

Since $G_n \times I_n$ converges to $X \times X$, the sequence of continua $Z_n$ accumulates on a continuum $Z \subset X \times X$.  Clearly $d(x,y) \geq \varepsilon$ for all $(x,y) \in Z$, and the second coordinate projection of $Z$ is $X$ since the second coordinate projection of $Z_n$ is $I_n$ for each $n = 1,2,\ldots$.  This means that $Z \cap \Delta X = \emptyset$ and $\pi_1(Z) \subseteq \pi_2(Z) = X$, hence $X$ does not have span zero, a contradiction.
\end{proof}

\section{Simple folds}
\label{sec:simple folds}

Throughout the remainder of this paper, $G$ will denote a (not necessarily connected) graph.  A subset $A$ of $G$ will be called \emph{regular} if $A$ is closed and has finitely many components, each of which is non-degenerate.  Note that a regular set always has finite boundary.

The following definition is adapted from \cite{OT86}.

\begin{defn}
\label{defn:simple fold}
A \emph{simple fold} on $G$ is a graph $F = F_1 \cup F_2 \cup F_3$ and a function $\varphi: F \to G$, called the \emph{projection}, which satisfy the following properties, where $G_i = \varphi(F_i)$ for $i = 1,2,3$:
\begin{enumerate}[label=(\textbf{F\arabic{*}})]
\item \label{enum:regular} $G_1$, $G_2$, and $G_3$ are non-empty regular subsets of $G$;
\item \label{enum:G1 G3} $G_1 \cup G_3 = G$, and $G_2 = G_1 \cap G_3$;
\item \label{enum:separate} $\overline{G_1 \smallsetminus G_2} \cap \overline{G_3 \smallsetminus G_2} = \emptyset$;
\item \label{enum:pw homeo} $\varphi {\upharpoonright}_{F_i}$ is a homeomorphism of $F_i$ onto $G_i$ for each $i = 1,2,3$; and
\item \label{enum:intersect bdy} $\partial G_1 = \varphi(F_1 \cap F_2)$, $\partial G_3 = \varphi(F_2 \cap F_3)$, and $F_1 \cap F_3 = \emptyset$.
\end{enumerate}
\end{defn}

Observe that property \ref{enum:separate} implies that $G_1 \cap G_3 \subseteq G_2$, so in \ref{enum:G1 G3} we could replace the condition $G_2 = G_1 \cap G_3$ with $G_2 \subseteq G_1 \cap G_3$.

See Figure \ref{fig:simple fold} for two examples of simple folds, in which both graphs $F$ and $G$ are arcs.

\begin{figure}
\begin{center}
\includegraphics{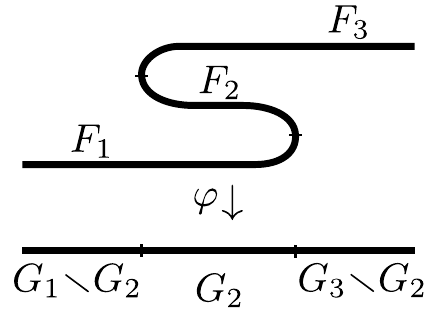}
\hspace{0.35in}
\includegraphics{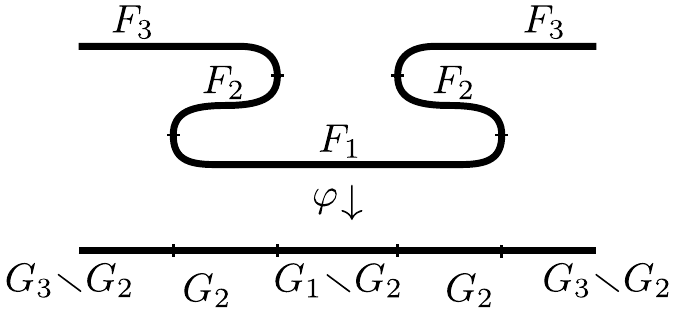}
\end{center}

\caption{Two examples of simple folds $\varphi: F \to G$, where $F$ and $G$ are arcs.  In both cases, the map $\varphi$ is the vertical projection.  Note that in the second example, the sets $F_2$, $F_3$, $G_2$, and $G_3$ are all disconnected (each has two components).}
\label{fig:simple fold}
\end{figure}

We record here some basic properties of simple folds.  The proofs of these properties are left to the reader.

\begin{lem}
\label{lem:fold basic}
Let $F = F_1 \cup F_2 \cup F_3$ be a simple fold on $G$ with projection $\varphi: F \to G$, and let $G_i = \varphi(F_i)$ for $i = 1,2,3$.  Then
\begin{enumerate}[label=(\arabic{*})]
\item \label{enum:fold boundary 1} $\partial G_2 = \partial G_1 \cup \partial G_3$ and $\partial G_1 \cap \partial G_3 = \emptyset$;
\item \label{enum:fold boundary 2} $\partial (G_1 \smallsetminus G_2) = \partial G_3$ and $\partial (G_3 \smallsetminus G_2) = \partial G_1$;
\item $F_1$, $F_2$, and $F_3$ are regular subsets of $F$;
\item $F_1 \cap F_2$ and $F_2 \cap F_3$ are finite sets;
\item \label{enum:creases} $\partial F_1 = F_1 \cap F_2$, $\partial F_3 = F_2 \cap F_3$, and $\partial F_2 = \partial F_1 \cup \partial F_3$;
\item \label{enum:inter sep} $\partial F_1$ separates $F_1 \smallsetminus \partial F_1$ from $(F_2 \cup F_3) \smallsetminus \partial F_1$ in $F$, and $\partial F_3$ separates $F_3 \smallsetminus \partial F_3$ from $(F_1 \cup F_2) \smallsetminus \partial F_3$ in $F$; and
\item \label{enum:open map} $\varphi {\upharpoonright}_{F_i \smallsetminus \partial F_i}: F_i \smallsetminus \partial F_i \to G$ is an open map for each $i = 1,2,3$.
\end{enumerate}
\end{lem}

To define a simple fold, it is enough to identify three subsets $G_1,G_2,G_3$ of $G$ satisfying properties \ref{enum:regular}, \ref{enum:G1 G3}, and \ref{enum:separate}.  Indeed, take spaces $E_1,E_2,E_3$ with $E_i \approx G_i$ and take homeomorphisms $\varphi_i: E_i \to G_i$.  Define $F = (E_1 \sqcup E_2 \sqcup E_3)/{\sim}$, where $\sim$ identifies pairs of the form $p \in E_i$, $q \in E_2$ with $\varphi_i(p) = \varphi_2(q) \in \partial G_i$ for $i = 1,3$.  Define $F_i$ to be the projection of $E_i$ in the quotient space $F$ and define $\varphi: F \to G$ by $\varphi {\upharpoonright}_{F_i} = \varphi_i$, for each $i = 1,2,3$.  It is straightforward to see that this is a well defined simple fold, and if $F'$ is another simple fold on $G$ with projection $\varphi'$ such that $\varphi'(F_i') = G_i$ for $i = 1,2,3$, then there is a homeomorphism $\theta: F' \to F$ with $\theta(F_i') = F_i$ for $i = 1,2,3$ and $\varphi' = \varphi \circ \theta$.


In general, even if $G$ is connected, a simple fold $F$ on $G$ need not be connected.  However, the next proposition shows that for connected $G$ we can always reduce $F$ to a connected simple fold.

Note that if $G$ is connected and $\partial G_1 = \emptyset$, then $G_1 = G$, $F$ is disconnected, and $\varphi {\upharpoonright}_{F_1}$ is a homeomorphism of $F_1$ onto $G$.  Likewise, if $\partial G_3 = \emptyset$, then $\varphi {\upharpoonright}_{F_3}$ is a homeomorphism of $F_3$ onto $G$.  In light of this, we will assume $\partial G_1 \neq \emptyset \neq \partial G_3$ in the following proposition.

\begin{prop}
\label{prop:conn fold}
Let $F = F_1 \cup F_2 \cup F_3$ be a simple fold on $G$ with projection $\varphi: F \to G$, and let $G_i = \varphi(F_i)$ for $i = 1,2,3$.  Suppose $G$ is connected, and that $\partial G_1 \neq \emptyset \neq \partial G_3$.  Then there is a component $C$ of $F$ such that $\varphi(C)$ meets $\partial G_1$ and $\partial G_3$.  Moreover, for any such component, $\varphi(C) = G$, and if we let $F_i' = F_i \cap C$ for $i = 1,2,3$, then $F' = F_1' \cup F_2' \cup F_3'$ is also a simple fold on $G$, with projection map $\varphi {\upharpoonright}_{F'}: F' \to G$.
\end{prop}

\begin{proof}
We first prove that there exists a component $C$ of $F$ such that $\varphi(C)$ meets $\partial G_1$ and $\partial G_3$.  By \ref{enum:G1 G3}, \ref{enum:separate} and since $G$ is connected, there is a component $K$ of $G_2$ which meets both $\overline{G_1 \smallsetminus G_2}$ and $\overline{G_3 \smallsetminus G_2}$.  By Lemma \ref{lem:fold basic}\ref{enum:fold boundary 2}, it follows that $K \cap \partial G_1 \neq \emptyset \neq K \cap \partial G_3$.  Because $\varphi {\upharpoonright}_{F_2}$ is a homeomorphism of $F_2$ onto $G_2$ (by \ref{enum:pw homeo}), we have that there is a component $C$ of $F$ such that $\varphi^{-1}(K) \cap F_2 \subset C$.  Then $\varphi(C) \supseteq K$, so $\varphi(C) \cap \partial G_1 \neq \emptyset \neq \varphi(C) \cap \partial G_3$.

Now fix any such component $C$ of $F$.

\begin{claim}
\label{claim:in C}
If $C' \subseteq C$ is any connected subset such that $\varphi(C') \subset G_2$ and $\varphi(C') \cap \partial G_1 \neq \emptyset \neq \varphi(C') \cap \partial G_3$, then $\varphi^{-1}(\varphi(C')) \subset C$.
\end{claim}

\begin{proof}[Proof of Claim \ref{claim:in C}]
\renewcommand{\qedsymbol}{\textsquare (Claim \ref{claim:in C})}
Let $C' \subseteq C$ be a connected subset such that $\varphi(C') \subset G_2$ and $\varphi(C') \cap \partial G_1 \neq \emptyset \neq \varphi(C') \cap \partial G_3$.  Since $G_2 = G_1 \cap G_3$ (by \ref{enum:G1 G3}), we have that $\varphi^{-1}(\varphi(C')) \cap F_1$, $\varphi^{-1}(\varphi(C')) \cap F_2$, and $\varphi^{-1}(\varphi(C')) \cap F_3$ are all homeomorphic to $\varphi(C')$ by \ref{enum:pw homeo}; in particular they are all connected.  Moreover, $\varphi^{-1}(\varphi(C')) \cap F_1 \cap F_2 \neq \emptyset$ since $\varphi(C') \cap \partial G_1 \neq \emptyset$ and $\partial G_1 = \varphi(F_1 \cap F_2)$ by \ref{enum:intersect bdy}.  Likewise, $\varphi^{-1}(\varphi(C')) \cap F_2 \cap F_3 \neq \emptyset$.  It follows that $\varphi^{-1}(\varphi(C'))$, which is the union of the three sets $\varphi^{-1}(\varphi(C')) \cap F_1$, $\varphi^{-1}(\varphi(C')) \cap F_2$, and $\varphi^{-1}(\varphi(C')) \cap F_3$, is connected.  Thus $\varphi^{-1}(\varphi(C')) \subset C$.
\end{proof}

Since $C$ is closed, $\varphi(C)$ is closed in $G$.  To show that $\varphi(C) = G$, we will show that $\varphi(C)$ is also open; this suffices since $G$ is connected.  To this end, let $x \in \varphi(C)$, and let $p \in C$ be such that $\varphi(p) = x$.  If $p \notin \partial F_2$, then by Lemma \ref{lem:fold basic}\ref{enum:open map} $\varphi$ is a homeomorphism in a neighborhood of $p$, so since $C$ is open in $F$, $\varphi(C)$ contains a neighborhood of $x$.

Suppose now that $p \in \partial F_2$.  Then $p \in \partial F_1 \cup \partial F_3$ by Lemma \ref{lem:fold basic}\ref{enum:creases}; say $p \in \partial F_1$.  Let $C'$ be the closure of the component of $C \smallsetminus \varphi^{-1}(\partial G_3)$ containing $p$.  Then by the Boundary Bumping Theorem (see e.g.\ \cite[Theorem 5.4]{nadler92}), $C' \cap \varphi^{-1}(\partial G_3) \neq \emptyset$.  Thus by Claim \ref{claim:in C}, we have $\varphi^{-1}(\varphi(C')) \subset C$.  In particular, the point $q = (\varphi {\upharpoonright}_{F_3})^{-1}(x) \in C$.  But $q \notin \partial F_3$ (because $\varphi(q) = x \in \partial G_1$ and by \ref{enum:intersect bdy} and Lemma \ref{lem:fold basic}\ref{enum:creases}, $\varphi(\partial F_3) = \partial G_3$, which is disjoint from $\partial G_1$), thus $q \notin \partial F_2$, and so again as above, $\varphi(C)$ contains a neighborhood of $\varphi(q) = x$.  The argument for $p \in \partial F_3$ is similar.

Therefore $\varphi(C) = G$.  It is straightforward to check from the definition of a simple fold that if $C \subset F$ is a component with $\varphi(C) = G$, then $F' = F_1' \cup F_2' \cup F_3'$, where $F_i' = F_i \cap C$ for $i = 1,2,3$, is a simple fold on $G$ with projection map $\varphi {\upharpoonright}_{F'}$ (note that it may well happen that $G'_i = \varphi(F'_i)$ is a proper subset of $G_i$ for one or more $i = 1,2,3$).
\end{proof}

The next result is related to Theorem 2 of \cite{OT86}, and it is alluded to in that paper though not treated in detail there.  It should be considered as a translation to the setting of simple folds of the following result of Krasinkiewicz and Minc \cite{KM77}: A continuum $X$ is hereditarily indecomposable if and only if for any disjoint closed subsets $A$ and $B$ of $X$ and any open sets $U$ and $V$ containing $A$ and $B$, respectively, there exist three closed sets $X_1,X_2,X_3 \subset X$ such that $X = X_1 \cup X_2 \cup X_3$, $A \subset X_1$, $B \subset X_3$, $X_1 \cap X_2 \subset V$, $X_2 \cap X_3 \subset U$, and $X_1 \cap X_3 = \emptyset$.  We remark that one can replace ``hereditarily indecomposable continuum'' with ``hereditarily indecomposable compactum'' in this result; the proof is unchanged.

\begin{thm}
\label{thm:factor fold}
Let $X$ be a compactum.  The following are equivalent:
\begin{enumerate}[label=(\arabic{*})]
\item \label{enum:factor fold 1} $X$ is hereditarily indecomposable
\item \label{enum:factor fold 2} For any map $f: X \to G$ to a graph $G$, for any simple fold $\varphi: F \to G$, and for any $\varepsilon > 0$, there exists a map $g: X \to F$ such that $\dsup(f, \varphi \circ g) < \varepsilon$
\item \label{enum:factor fold 3} For any map $f: X \to [0,1]$, for any simple fold $\varphi: F \to [0,1]$ where $F$ is an arc, and for any $\varepsilon > 0$, there exists a map $g: X \to F$ such that $\dsup(f, \varphi \circ g) < \varepsilon$.
\end{enumerate}
\end{thm}

\begin{proof}
To show \ref{enum:factor fold 1} $\Rightarrow$ \ref{enum:factor fold 2}, suppose that \ref{enum:factor fold 1} holds.  Let $G$ be a graph, $f: X \to G$ a map, $\varphi: F \to G$ a simple fold, and fix $\varepsilon > 0$.  As in Definition \ref{defn:simple fold}, denote $G_i = \varphi(F_i)$ for $i = 1,2,3$.

Suppose that $\partial (G_i \smallsetminus G_2) = \{y^i_1,\ldots,y^i_{m(i)}\}$ for $i = 1,3$.  Each of these points $y^i_j$ is the vertex point of a finite fan $Y^i_j \subset G_2$ (meaning $Y^i_j$ is the union of finitely many arcs, each having $y^i_j$ as one endpoint, and which are otherwise pairwise disjoint) such that $Y^i_j$ is the closure of an open neighborhood $O^i_J$ of $y^i_j$ in $G_2$, and the diameter of $Y^i_j$ is less than $\varepsilon$.  Let $K_i = f^{-1}(\overline{G_i \smallsetminus G_2})$.  Then $K_1 \cap K_3 = \emptyset$ by \ref{enum:separate} of Definition \ref{defn:simple fold}.

For $i = 1,3$, choose neighborhoods $U_i$ of $K_i$ so that $f(\overline{U_i \smallsetminus K_i}) \subset \bigcup_j O^i_j$.  By \cite{KM77}, there exist closed sets $X_i$, $i = 1,2,3$, such that
\begin{itemize}
\item $X = X_1 \cup X_2 \cup X_3$,
\item $K_i \subset X_i$ for $i = 1,3$,
\item $X_1 \cap X_3 = \emptyset$,
\item $X_1 \cap X_2 \subset U_3$, and
\item $X_2 \cap X_3 \subset U_1$.
\end{itemize}
See Figure \ref{fig:factor fold} for an illustration.

\begin{figure}
\begin{center}
\includegraphics{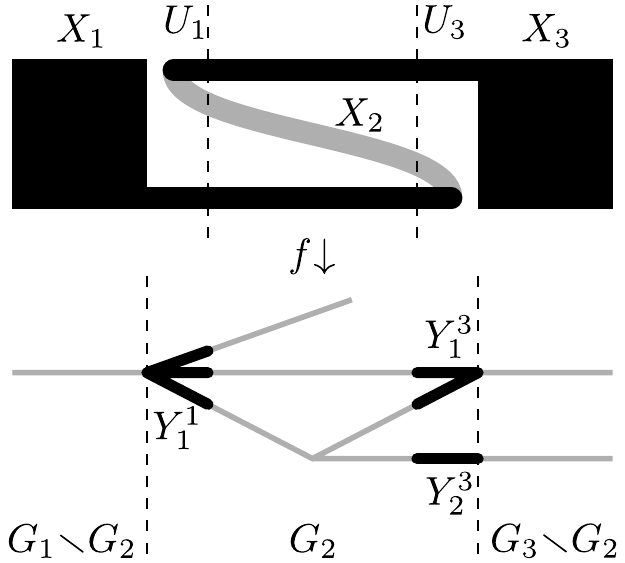}
\end{center}

\caption{An illustration of the situation in the proof of Theorem \ref{thm:factor fold}.}
\label{fig:factor fold}
\end{figure}

Let
\[ A = [X_1 \smallsetminus U_3] \cup [X_2 \smallsetminus (U_1 \cup U_3)] \cup [X_3 \smallsetminus U_1] ,\]
and consider the restriction $f {\upharpoonright}_A$.  Observe that
\[ X \smallsetminus A = [(X_2 \cup X_3) \cap U_1] \cup [(X_1 \cup X_2) \cap U_3] .\]

We extend $f {\upharpoonright}_A$ to a map $h: X \to G$ as follows.  Observe that $f((X_2 \cup X_3) \cap U_1) \subset \bigcup_j Y^1_j$.  In fact, for each $j = 1,\ldots,m(1)$, since $[(X_2 \cup X_3) \cap U_1] \cap X_1 = \emptyset$ and $f^{-1}(y^1_j) \subset f^{-1}(\overline{G_1 \smallsetminus G_2}) = K_1 \subset X_1$, we have that $f((X_2 \cup X_3) \cap U_1) \subset \bigcup_j (Y^1_j \smallsetminus \{y^1_j\})$.  Let $L$ be an arc in $Y^1_j$ with one endpoint $y^1_j$ and the other endpoint equal to an endpoint of the fan $Y^1_j$.  Let $L' =  f^{-1}(L) \cap (X_2 \cup X_3) \cap U_1$.  Then $L'$ is closed and open in $(X_2 \cup X_3) \cap U_1$.  By the Tietze extension theorem, we can define a continuous function $h_L: \overline{L'} \to L$ so that $h_L {\upharpoonright}_{\partial L'} = f {\upharpoonright}_{\partial L'}$ and $h_L(x) = y^1_j$ for all $x \in X_2 \cap X_3 \cap L'$.  We then let $h {\upharpoonright}_{\overline{L'}} = h_L$, and do this for all such arcs $L$ in the fans $Y^1_j$.  We proceed similarly to define $h$ on $(X_1 \cup X_2) \cap U_3$.

In this way, we obtain a continuous function $h: X \to G$ such that
\begin{itemize}
\item $h {\upharpoonright}_A = f {\upharpoonright}_A$,
\item $h(X_i) \subseteq G_i$ for $i = 1,2,3$,
\item $h(X_1 \cap X_2) \subset \partial(G_3 \smallsetminus G_2) = \partial G_1$ (see Lemma \ref{lem:fold basic}\ref{enum:fold boundary 2}), and
\item $h(X_2 \cap X_3) \subset \partial(G_1 \smallsetminus G_2) = \partial G_3$ (see Lemma \ref{lem:fold basic}\ref{enum:fold boundary 2}).
\end{itemize}
Observe that $\dsup(f,h) < \varepsilon$ since the diameters of the sets $Y^i_j$ are less than $\varepsilon$.

Now define $g: X \to F$ by $g(x) = \left( (\varphi {\upharpoonright}_{F_i})^{-1} \circ h \right) (x)$ if $x \in X_i$, for $i = 1,2,3$.  This is well-defined and continuous because of the above properties of $h$ and of $X_1,X_2,X_3$.  Then $g$ is as required so that \ref{enum:factor fold 2} holds.

The implication \ref{enum:factor fold 2} $\Rightarrow$ \ref{enum:factor fold 3} is trivial.

To show \ref{enum:factor fold 3} $\Rightarrow$ \ref{enum:factor fold 1}, suppose that \ref{enum:factor fold 3} holds.  Let $A,B \subset X$ be disjoint closed sets, and let $U$ be a neighborhood of $A$ and $V$ a neighborhood of $B$.  By \cite{KM77} it suffices to show that $X = X_1 \cup X_2 \cup X_3$ where $X_1,X_2,X_3$ are closed subsets of $X$ so that $A \subset X_1$, $B \subset X_3$, $X_1 \cap X_3 = \emptyset$, $X_1 \cap X_2 \subset U$, and $X_2 \cap X_3 \subset V$.

Let $f: X \to [0,1]$ be a map such that $f^{-1}(0) = A$ and $f^{-1}(1) = B$.  Choose $0 < u < v < 1$ such that $f^{-1}([0,u]) \subset U$ and $f^{-1}([v,1]) \subset V$.  Let $u' \in (0,u)$ and $v' \in (v,1)$.  Construct a simple fold $\varphi: F \to [0,1]$, where $F = F_1 \cup F_2 \cup F_3$ is an arc, such that $\varphi(F_1) = [0,v']$, $\varphi(F_2) = [u',v']$ and $\varphi(F_3) = [u',1]$.  Let $\varepsilon > 0$ be small enough so that $(u' - \varepsilon, u' + \varepsilon) \subset [0,u)$ and $(v' - \varepsilon, v' + \varepsilon) \subset [v,1]$.  By (3), there is a map $g: X \to F$ such that $\dsup(f, \varphi \circ g) < \varepsilon$.

Put $X_i = g^{-1}(F_i)$ for $i = 1,2,3$.  Then $X = X_1 \cup X_2 \cup X_3$, and clearly $X_1 \cap X_3 = \emptyset$.  To see that $X_1 \cap X_2 \subset V$, let $x \in X_1 \cap X_2$.  Then $(\varphi \circ g)(x) = v'$, and since $\dsup(f,\varphi \circ g) < \varepsilon$, we have $f(x) \in [v,1]$ and, hence, $x \in V$.  Similarly, $X_2 \cap X_3 \subset U$.  By \cite{KM77}, $X$ is hereditarily indecomposable.
\end{proof}

\bigskip
We now introduce notions which will be relevant when considering structured separators in the next section.

\begin{defn}
\label{defn:cons comp}
Let $A \subset G$ be regular, and let $B \subset \partial A$.
\begin{itemize}
\item $A$ has \emph{consistent complement relative to $B$} if for each component $C$ of $G \smallsetminus A$, either $\partial C \subseteq B$ or $\partial C \cap B = \emptyset$.
\item The \emph{$A$ side of $B$}, denoted $\sigma_B(A)$, is the closure of the union of all components of $G \smallsetminus B$ meeting $A$.
\end{itemize}
\end{defn}

If $B$ is empty, then $\sigma_B(A)$ is simply equal to the union of all components of $G$ which $A$ intersects.  In particular, if $A = \emptyset$ then $\sigma_B(A) = \emptyset$.

Suppose that $A$ and $B$ are both non-empty.  Observe that if $G$ is connected and $A$ has consistent complement relative to $B$, then in fact for any neighborhood $V$ of $B$, $\sigma_B(A)$ is equal to the closure of the union of all components of $G \smallsetminus B$ meeting $A \cap V$.  In fact, $\sigma_B(A)$ can be characterized as the unique closed (regular) set $D \subset G$ such that $\partial D = B$ and $D \cap V = A \cap V$ for some neighborhood $V$ of $B$.

\begin{prop}
\label{prop:side eq}
Let $G$ be connected, let $A,A' \subset G$ be non-empty regular sets, and let $B \subseteq \partial A \cap \partial A'$.  If $A$ and $A'$ each have consistent complement relative to $B$, and if there is a neighborhood $V$ of $B$ such that $A \cap V = A' \cap V$, then $\sigma_B(A) = \sigma_B(A')$.  Moreover, if $C$ is a component of $G \smallsetminus A$ with $\overline{C} \cap B \neq \emptyset$, then $C$ is also a component of $G \smallsetminus A'$.
\end{prop}

\begin{proof}
That $\sigma_B(A) = \sigma_B(A')$ follows immediately from the observations after Definition \ref{defn:cons comp}.  For the moreover part, let $C$ be a component of $G \smallsetminus A$ with $\overline{C} \cap B \neq \emptyset$.  Since $C \cap A = \emptyset$ and $A \cap V = A' \cap V$ (where $V$ is the neighborhood of $B$ described in the statement of this proposition), we have $(C \cap V) \cap A' = \emptyset$.  Let $C'$ be a component of $G \smallsetminus A'$ meeting $C \cap V$.

Obviously $C' \subseteq C$, since $\partial C \subseteq B$ and $C' \cap B = \emptyset$.  If $C' \neq C$, then there must be a point $x \in \partial C' \cap C$.  But since $A'$ has consistent complement relative to $B$, we must have $x \in B$, so $\emptyset \neq C \cap B \subset C \cap A$, a contradiction.  Therefore $C' = C$.
\end{proof}

\begin{prop}
\label{prop:define fold}
Let $G$ be connected, let $A \subset G$ be regular and non-empty, and let $B_1,B_2 \subseteq \partial A$ with $B_1 \cup B_2 = \partial A$ and $B_1 \cap B_2 = \emptyset$.  Suppose $A$ has consistent complement relative to $B_1$ and to $B_2$.  Let $G_1 = \sigma_{B_1}(A)$, $G_2 = A$, and $G_3 = \sigma_{B_2}(A)$.  Then $G_1$, $G_2$, and $G_3$ define a simple fold on $G$ (i.e.\ they satisfy properties \ref{enum:regular}, \ref{enum:G1 G3}, and \ref{enum:separate}).
\end{prop}

\begin{proof}
Note that if $A = G$, then $G_1 = G_2 = G_3 = G$, which define a simple fold.  We suppose therefore that $A \neq G$, in which case at least one of $B_1$ and $B_2$ is non-empty.

Clearly $G_1$, $G_2$, and $G_3$ are all regular subsets of $G$, so \ref{enum:regular} holds.

Consider \ref{enum:G1 G3}.  By definition, it is clear that $A \subseteq \sigma_{B_1}(A)$ and $A \subseteq \sigma_{B_2}(A)$, thus $G_2 \subseteq G_1 \cap G_3$.  For the reverse inclusion, suppose $x \in G \smallsetminus G_2 = G \smallsetminus A$, and let $C$ be the component of $G \smallsetminus A$ containing $x$.  Because $G$ is connected, $\overline{C} \cap A \neq \emptyset$, and either $\partial C \subseteq B_1$ or $\partial C \subseteq B_2$ since $A$ has consistent complement relative to $B_1$ and to $B_2$.  In the former case, we have $C \cap \sigma_{B_1}(A) = \emptyset$, and in the latter case we have $C \cap \sigma_{B_2}(A) = \emptyset$.  In either case, $x \notin G_1 \cap G_3$.  Thus $G_1 \cap G_3 \subseteq G_2$.

To see that $G_1 \cup G_3 = G$, let $x \in G$, and assume $x \notin A$ (since $A = G_2 = G_1 \cap G_3$).  Let $C$ be the component of $G \smallsetminus A$ containing $x$.  Again $\overline{C} \cap A \neq \emptyset$, and either $\partial C \subseteq B_1$ or $\partial C \subseteq B_2$.  If $\partial C \subseteq B_1$, then since $\sigma_{B_2}(A) \supset A \supset B_1$, it is clear that $C \subset \sigma_{B_2}(A)$.  Similarly, if $\partial C \subseteq B_2$, then $C \subset \sigma_{B_1}(A)$.  Thus in any case, $x \in G_1 \cup G_3$.

For \ref{enum:separate}, let $x \in \overline{G_1 \smallsetminus G_2}$.  If $x \in A$, then we must have $x \in B_2$, and in this case $x \notin \overline{\sigma_{B_2}(A) \smallsetminus A} = \overline{G_3 \smallsetminus G_2}$, since one can find a neighborhood of $x$ which meets only $A$ and components of $G \smallsetminus A$ whose closures meet $B_2$.  On the other hand, if $x \notin A$, then $x \notin \sigma_{B_2}(A) = G_3$ because $x \in \sigma_{B_1}(A)$ and $\sigma_{B_1}(A) \cap \sigma_{B_2}(A) = A$.  Thus in any case, $x \notin \overline{G_3 \smallsetminus G_2}$.  Therefore $\overline{G_1 \smallsetminus G_2} \cap \overline{G_3 \smallsetminus G_2} = \emptyset$.
\end{proof}

We remark that if $\partial A = B_1 \cup B_2$ and $B_1 \cap B_2 = \emptyset$, and if $A$ has consistent complement relative to $B_1$, then $A$ automatically has consistent complement relative to $B_2$ as well.

\section{Stairwells}
\label{sec:stairwells}

We pause here to give an outline of the remainder of the proof of Theorem \ref{thm:hered indec span zero}, which is presented in full in Section \ref{sec:applications}.  Beginning with a hereditarily indecomposable continuum $X$ with span zero, in the Hilbert cube $[0,1]^{\mathbb{N}}$, we fix some $\varepsilon > 0$ and let $I \approx [0,1]$ be an arc which is close to $X$ in Hausdorff distance.  Our task is to produce an $\varepsilon$-map from $X$ to $I$, which would imply $X$ is arc-like, and hence $X$ is homeomorphic to the pseudo-arc by Bing's characterization \cite{bing51}.

Because $X$ has span zero, it is tree-like \cite{lelek79}, so we can choose a tree $T \subset [0,1]^{\mathbb{N}}$ and a map $f: X \to T$ such that $d(x,f(x))$ is small for all $x \in X$.  According to Theorem \ref{thm:span separator}, the set $M = \{(x,y) \in T \times I: d(x,y) < \frac{\varepsilon}{6}\}$ separates $T \times \{0\}$ from $T \times \{1\}$ in $T \times I$ provided $T$ and $I$ are chosen close enough to $X$.  If we can find a map $h: X \to M$, $h(x) = (h_1(x),h_2(x))$, such that $d(h_1(x),f(x))$ is small for all $x \in X$, then by the definition of $M$ and choice of $f$, it follows that $h_2(x)$ is close to $x$ for all $x \in X$, and so $h_2$ will be an $\varepsilon$-map once appropriate care is taken with constants.

To obtain this map $h: X \to M$, we use our assumption that $X$ is hereditarily indecomposable.  According to Theorem \ref{thm:factor fold}, the map $f: X \to T$ can be (approximately) factored through any simple fold $\phi: F \to T$.  Our method is to inspect the structure of the separator $M$ and use a sequence of simple folds to match the continuum $X$ and map $f$ with the pattern of $M$ and the first coordinate projection $\pi_1$.

In order to accomplish this, we introduce in this section a special type of separator (one with a ``stairwell structure'') which has a positive integer measure of complexity (the ``height'' of the stairwell).  It follows from Theorem \ref{thm:get stairwell} below that $M$ contains a subset which is a separator with a stairwell structure.  We then prove in the next section that one can use a sequence of simple folds to effectively reduce the height of a stairwell.  The proof of Theorem \ref{thm:hered indec span zero} is then completed by induction (note from Definition \ref{defn:stairwell} below that if $S$ has a stairwell structure of height $1$, then $\pi_1 {\upharpoonright}_S$ is one-to-one and, hence, a homeomorphism).

\bigskip
Given a set $X$, let $X_\star = X \times [0,1]$.  Define $\pi: X_\star \to X$ by $\pi(x,t) = x$.  Given a function $f: X \to Y$, define $f_\star: X_\star \to Y_\star$ by $f_\star(x,t) = (f(x),t)$.

\begin{defn}
\begin{itemize}
\item A collection $\langle B_1,\ldots,B_n \rangle$ of finite subsets of $G$ is \emph{generic} if $B_i$ is disjoint from the set of branch points and endpoints of $G$ for each $i$, and $B_i \cap B_j = \emptyset$ whenever $i \neq j$.
\item A subset $S \subset G_\star$ is \emph{straight} if $S$ is closed, $\pi$ is one-to-one on $S$, and $\pi(S)$ is regular.  The \emph{end set} of a straight subset $S \subset G_\star$ is $\setends{S} = S \cap \pi^{-1}(\partial \pi(S))$.
\end{itemize}
\end{defn}

See Figure \ref{fig:straight} for an example of a straight set and its end set.

Observe that if $S \subset G_\star$ is straight then $\pi$, restricted to $S \smallsetminus \setends{S}$, is an open mapping from $S \smallsetminus \setends{S}$ to $G$ (see Figure \ref{fig:straight}).

\begin{figure}
\begin{center}
\includegraphics{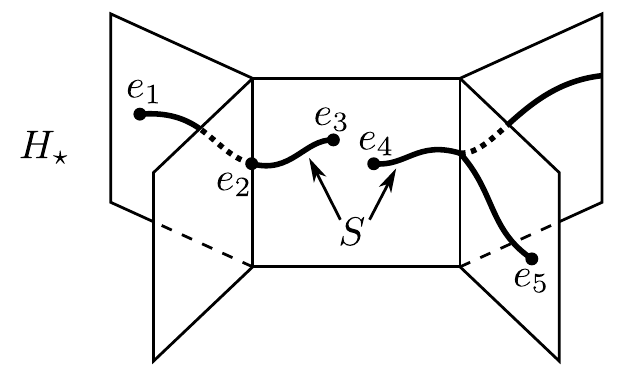}
\end{center}

\caption{An example of a straight set $S$ in $H_\star = H \times [0,1]$, where $H$ is a graph homeomorphic to the letter ``$\mathsf{H}$''.  In this example, $S$ has two connected components.  The end set of $S$ is $\setends{S} = \{e_1,e_2,e_3,e_4,e_5\}$.}
\label{fig:straight}
\end{figure}

\begin{defn}
\label{defn:stairwell}
Let $S \subset G_\star$.  A \emph{stairwell structure for $S$ of height $k$} is a tuple $\langle S_1,\ldots,S_k \rangle$ such that:
\begin{enumerate}[label=\textbf{(S\arabic{*})}]
\item \label{enum:straight} $S_1,\ldots,S_k$ are non-empty straight subsets of $G_\star$ with $S = S_1 \cup \cdots \cup S_k$;
\item \label{enum:ends decomp} For each $i = 1,\ldots,k$, $\setends{S_i} = \alpha_i \cup \beta_i$, where $\alpha_i$ and $\beta_i$ are disjoint finite sets, $\alpha_1 = \emptyset = \beta_k$, and $\beta_i = \alpha_{i+1}$ for each $i = 1,\ldots,k-1$;
\item \label{enum:same side} For each $i = 1,\ldots,k-1$ there is a neighborhood $V$ of $\pi(\beta_i) = \pi(\alpha_{i+1})$ such that $\pi(S_i) \cap V = \pi(S_{i+1}) \cap V$;
\item \label{enum:consistent} For each $i = 1,\ldots,k$, $\pi(S_i)$ has consistent complement relative to $\pi(\alpha_i)$ and to $\pi(\beta_i)$;
\item \label{enum:generic} The family $\langle \pi(\alpha_2),\ldots,\pi(\alpha_k) \rangle$ (which is equal to $\langle \pi(\beta_1),\ldots,\pi(\beta_{k-1}) \rangle$) is generic in $G$.
\end{enumerate}
\end{defn}

See Figure \ref{fig:stairwell} for a simple example of a set with a stairwell structure.

Note that even though the sets $S_1,\ldots,S_k$ are all non-empty, we do allow for the possibility that $\alpha_i = \emptyset$ for some values of $i \in \{2,\ldots,k\}$.

\begin{figure}
\begin{center}
\includegraphics{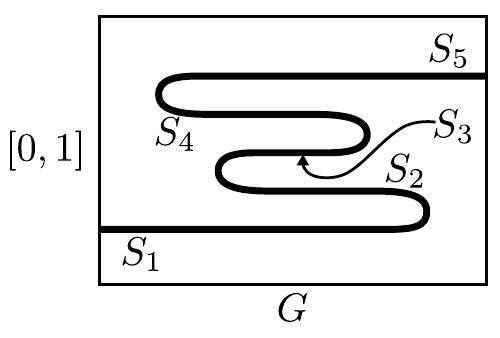}
\end{center}

\caption{An example of a set with a stairwell structure of height $5$ in $G_\star$, where $G$ is an arc.}
\label{fig:stairwell}
\end{figure}

We make the following observation: if $S \subset G_\star$ is a set with a stairwell structure, and if $C$ is a component of $G$ such that $S_i \cap C_\star \neq \emptyset$ for each $i = 1,\ldots,k$, then $S \cap C_\star$ has a stairwell structure obtained by intersecting each of the sets $S_i\alpha_i,\beta_i,$ with $C_\star$.

Note that there is no requirement that a set $S \subset G_\star$ with a stairwell structure of height $k$ will satisfy $\pi(S) = G$.  Indeed, if $k$ is even this need not be the case.  However, it will follow from the next proposition (in fact from Claim \ref{claim:ceiling parity}) that if $k$ is odd then $\pi(S) = G$.  Here it is crucial that $\alpha_1 = \beta_k = \emptyset$ (see \ref{enum:ends decomp}).  The reader is encouraged to draw a couple of examples of sets with stairwell structures of even and odd heights in $G_\star$, for $G$ a simple graph such as an arc, circle, or simple triod, to explore these possibilities.

Though we will not technically need the next proposition in the sequel, it serves to clarify the connection between separators in $G_\star$ and sets with stairwell structures.

\begin{prop}
\label{prop:odd separate}
If $G$ is a connected graph, then a set $S \subset G \times (0,1)$ with a stairwell structure of odd height separates $G \times \{0\}$ from $G \times \{1\}$ in $G_\star$.
\end{prop}

\begin{proof}
Let $\langle S_1,\ldots,S_k \rangle$ be a stairwell structure for $S$, where $k$ is odd.

\begin{claim}
\label{claim:ceiling parity}
For each $x \in G$, the number of integers $i \in \{1,\ldots,k\}$ such that $x \in \pi(S_i)$ is odd.
\end{claim}

\begin{proof}[Proof of Claim \ref{claim:ceiling parity}]
\renewcommand{\qedsymbol}{\textsquare (Claim \ref{claim:ceiling parity})}
Fix $x \in G$, and define $f: \{0,\ldots,k\} \to \{0,1\}$ by
\[ f(i) =
\begin{cases}
1 &\textrm{if } i = 0 \textrm{ or } i = k \textrm{ or } x \in \sigma_{\pi(\beta_i)}(\pi(S_i)) \\
0 &\textrm{ otherwise.}
\end{cases} \]

For each $i = 2,\ldots,k$, by property \ref{enum:same side}, Proposition \ref{prop:side eq}, and the fact that $\beta_{i-1} = \alpha_i$, we have that $\sigma_{\pi(\beta_{i-1})}(\pi(S_{i-1})) = \sigma_{\pi(\beta_{i-1})}(\pi(S_i)) = \sigma_{\pi(\alpha_i)}(\pi(S_i))$.  By Proposition \ref{prop:define fold} and property \ref{enum:G1 G3}, it follows that $\sigma_{\pi(\beta_{i-1})}(\pi(S_{i-1})) \cup \sigma_{\pi(\beta_i)}(\pi(S_i)) = G$ for each $i = 2,\ldots,k-1$.  This means that there are no contiguous blocks of more than one integer in $f^{-1}(0)$.  Observe that $x \in \pi(S_i)$ if and only if $f(i-1) = f(i) = 1$.  It follows that if $N_1$ is the number of integers $i \in \{1,\ldots,k\}$ such that $x \in \pi(S_i)$ and $N_2$ is the number of integers $i \in \{1,\ldots,k\}$ such that $f(i-1) \neq f(i)$, then $N_1 + N_2 = k$.

Since $f(0) = f(k) = 1$, we have that $N_2$ is even.  By hypothesis, $k$ is odd.  Thus $N_1$ must be odd.
\end{proof}

Given $(x,t) \in G_\star \smallsetminus S$, define $N(x,t) =$ the cardinality of the set of integers $i \in \{1,\ldots,k\}$ such that $(x,s) \in S_i$ for some $s > t$.  Let $V_1 = \{(x,t) \in G_\star \smallsetminus S: N(x,t)$ is odd$\}$ and $V_2 = \{(x,t) \in G_\star \smallsetminus S: N(x,t)$ is even$\}$.  From Claim \ref{claim:ceiling parity}, we have $G \times \{0\} \subset V_1$, and clearly $G \times \{1\} \subset V_2$ and $V_1 \cup V_2 = G_\star \smallsetminus S$.

\begin{claim}
\label{claim:separation}
$V_1$ and $V_2$ are open in $G_\star \smallsetminus S$.
\end{claim}

\begin{proof}[Proof of Claim \ref{claim:separation}]
\renewcommand{\qedsymbol}{\textsquare (Claim \ref{claim:separation})}
Fix $(x,t) \in V_1$.  Let $W$ be a small connected open neighborhood of $x$ in $G$, and let $\delta > 0$ be such that $U = W \times (t-\delta,t+\delta)$ is a neighborhood of $(x,t)$ in $G_\star$ which is disjoint from $S$.

If $x \notin \pi(\setends{S_i})$ for each $i$, then we may assume $W$ is small enough so that for each $i$, either $W \cap \pi(S_i) = \emptyset$ or $W \subset \pi(S_i)$.  It follows easily that for each $(x',t') \in U$, $N(x',t') = N(x,t)$.  Thus $U \subset V_1$.

If $x \in \pi(\setends{S_i})$ for some $i$, say $x \in \pi(\beta_i)$, then by \ref{enum:generic} $x \notin \pi(\setends{S_j})$ for each $j \notin \{i,i+1\}$, and so we may assume $W$ is small enough so that for each $j \notin \{i,i+1\}$, either $W \cap \pi(S_j) = \emptyset$ or $W \subset \pi(S_j)$.  Moreover, we may assume $W$ is small enough so that $W \cap \pi(S_i) = W \cap \pi(S_{i+1})$.

If there is no $s > t$ such that $(x,s) \in S_i$, then it is easy to see that $N(x',t') = N(x,t)$ for all $(x',t') \in U$.  Suppose then that there exists $s > t$ such that $(x,s) \in S_i$ (so that $(x,s) \in \beta_i$).  Let $(x',t') \in U$.  If $x' \in \pi(S_i)$ then $x' \in \pi(S_{i+1})$ as well, and it is clear that $N(x',t') = N(x,t)$.  If $x' \notin \pi(S_i)$, then $x' \notin \pi(S_{i+1})$ as well, and so $N(x',t') = N(x,t) - 2$.  In any case, we have $(x',t') \in V_1$.  Thus $U \subset V_1$.

Therefore $V_1$ is open.  The proof that $V_2$ is open is identical.
\end{proof}

Thus $S$ separates $G \times \{0\}$ from $G \times \{1\}$ in $G_\star$.
\end{proof}

As a special case, consider a set $S \subset G \times (0,1)$ with a stairwell structure of height $1$.  In this case, $\pi$ maps $S$ homeomorphically onto $G$.

\begin{thm}
\label{thm:get stairwell}
Let $G$ be a graph.  Given any set $M \subseteq G \times (0,1)$ which separates $G \times \{0\}$ from $G \times \{1\}$ in $G_\star$ and any open set $U \subseteq G \times (0,1)$ with $M \subseteq U$, there exists a set $S \subset U$ with a stairwell structure of odd height.
\end{thm}

\begin{proof}
Let $M \subset G \times (0,1)$ separate $G \times \{0\}$ from $G \times \{1\}$ in $G_\star$, and fix an open set $U \subseteq G \times (0,1)$ with $M \subseteq U$.

We say a set $S \subset G \times (0,1)$ \emph{irreducibly separates} $G \times \{0\}$ from $G \times \{1\}$ in $G_\star$ if $S$ separates these two sets, but no proper subset of $S$ does.  It is well known (see e.g.\ \cite[Theorems \S 46.VII.3 and \S 49.V.3]{kuratowski68}) that for any set $S \subset G \times (0,1)$ which separates $G \times \{0\}$ from $G \times \{1\}$, there is a closed set $S' \subseteq S$ which irreducibly separates $G \times \{0\}$ from $G \times \{1\}$.

Let $Z$ denote the set of all branch points and endpoints of $G$.  Given a set $L \subset G_\star$ and a point $(x,y) \in L$ such that $x \notin Z$, we say \emph{$L$ has a side wedge} at $(x,y)$ if there is a closed disk $D$ containing $(x,y)$ in its interior such that $L \cap D = C_1 \cup C_2$, where $C_1$ and $C_2$ are arcs which both have $x$ as an endpoint but are otherwise disjoint, $\pi$ is one-to-one on $C_1$ and on $C_2$, and $\pi(C_1) = \pi(C_2)$.

\begin{claim}
\label{claim:simple sep}
There exists a set $M' \subset U$ such that:
\begin{enumerate}[label=(\arabic{*})]
\item \label{enum:ss graph} $M'$ is a graph;
\item \label{enum:ss irred} $M'$ irreducibly separates $G \times \{0\}$ from $G \times \{1\}$ in $G_\star$;
\item \label{enum:ss T} There is a finite set $T \subset M'$ such that for all $(x,y) \in M' \smallsetminus T$, there is a neighborhood $V$ of $(x,y)$ such that $\pi$ maps $M' \cap V$ homeomorphically onto a neighborhood of $x$ in $G$;
\item \label{enum:ss wedge} For each $(x,y) \in T$, $M'$ has a side wedge at $(x,y)$;
\item \label{enum:ss turn pts 1} $T \cap Z_\star = \emptyset$; and
\item \label{enum:ss turn pts 2} If $(x_1,y_1)$ and $(x_2,y_2)$ are two distinct points in $T$, then $x_1 \neq x_2$.
\end{enumerate}
\end{claim}

\begin{proof}[Proof of Claim \ref{claim:simple sep}]S
\renewcommand{\qedsymbol}{\textsquare (Claim \ref{claim:simple sep})}
We leave it to the reader to show that there exists a set $M'$ having properties \ref{enum:ss graph}, \ref{enum:ss T}, \ref{enum:ss turn pts 1}, and \ref{enum:ss turn pts 2}, and which separates $G \times \{0\}$ from $G \times \{1\}$.  Replacing $M'$ by a subset (which, by abuse of notation, we also denote by $M'$) which irreducibly separates $G \times \{0\}$ from $G \times \{1\}$ in $G_\star$ accomplishes \ref{enum:ss irred}.  Let $G_\star \smallsetminus M' = R_0 \cup R_1$, where $R_0$ and $R_1$ are open in $G_\star \smallsetminus M'$, $G \times \{0\} \subset R_0$, and $G \times \{1\} \subset R_1$.

To achieve property \ref{enum:ss wedge}, consider a point $(x,y) \in T$.  Note that $(x,y)$ cannot be an endpoint of $M'$, because $x$ is not an endpoint of $G$ by \ref{enum:ss turn pts 1}, and $\pi(M' \cap V)$ is open for some neighborhood $V$ of $(x,y)$ by \ref{enum:ss T}.  If $(x,y)$ is not a branch point of $M'$, then it is easy to see that $M'$ has a side wedge at $(x,y)$, or else $\pi$ is one-to-one on $M'$ in a neighborhood of $(x,y)$ in which case we can remove $(x,y)$ from $T$. Again, by abuse of notation, we denote the resulting set by $M'$.

Suppose now that $(x,y)$ is a branch point of $M'$.  Let $D$ be a small closed disk containing $(x,y)$ in its interior such that $M' \cap D$ is the union of $n$ arcs $C_1,\ldots,C_n$, each having $(x,y)$ as an endpoint, and which are otherwise pairwise disjoint.  Because $M'$ is an irreducible separator, the complementary regions of $M'$ in $D$ alternate between $R_0$ and $R_1$.  It follows that $n$ is even.  Now we can modify $M'$ inside $D$ by replacing the arcs $C_1,\ldots,C_n$ with $\frac{n}{2}$ ``wedges'', as depicted in Figure \ref{fig:make simple}, and removing $(x,y)$ from $T$.  Some of the resultant wedges may be side wedges, whose ``tip'' points we add to $T$.  Obviously this can be done without compromising properties \ref{enum:ss graph}, \ref{enum:ss turn pts 1}, and \ref{enum:ss turn pts 2}, and without leaving $U$.

\begin{figure}
\begin{center}
\includegraphics{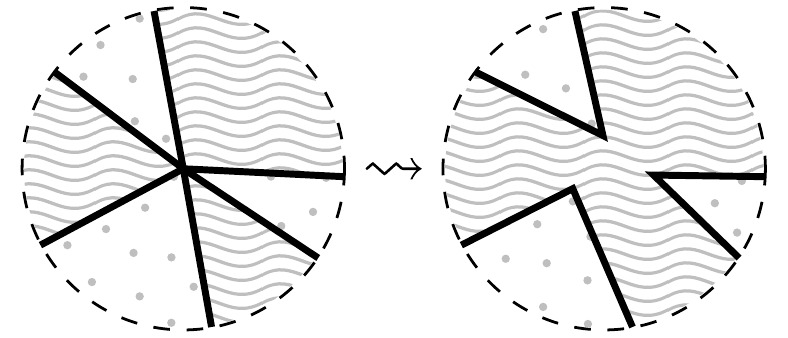}
\end{center}

\caption{Modifying a graph separator $M'$ in $G_\star$ in a small neighborhood of an unwanted branch point to remove the branch point.  The two sides, $R_0$ and $R_1$, of $G_\star \smallsetminus M'$ are indicated using wavy lines and dots, respectively.}
\label{fig:make simple}
\end{figure}

Once this is carried out for all the branch points of $M'$ which belong to $T$, one at a time, the resultant set satisfies property \ref{enum:ss wedge}.  It is easy to see that the resultant $M'$ still irreducibly separates $G \times \{0\}$ from $G \times \{1\}$ in $G_\star$.
\end{proof}

Given a finite set $B \subset G$, we say two points $a,b \in B$ are \emph{adjacent} if there is a component of $G \smallsetminus B$ whose closure contains both $a$ and $b$.

Let $M'$ be a set as described in Claim \ref{claim:simple sep}.  Because of property \ref{enum:ss turn pts 2}, there exists a finite set $Z' \subset G$ such that
\begin{itemize}
\item $Z \subseteq Z'$,  $Z' \cap \pi(T) = \emptyset$ and the closure of every component of $G \smallsetminus Z'$ is an arc;
\item If $a,b \in Z'$ are adjacent, then there is exactly one component of $G \smallsetminus Z'$ whose closure contains both $a$ and $b$ and  we will denote this arc by $[a,b]$); and
\item If $a,b \in Z'$ are adjacent, then $[a,b]_\star \cap T$ contains at most one point.
\end{itemize}

Figure \ref{fig:partition} illustrates what the set $M'\cap \pi^{-1}(A)$ might look like over some component
$A$ of $G \smallsetminus Z'$.

\begin{figure}
\begin{center}
\includegraphics{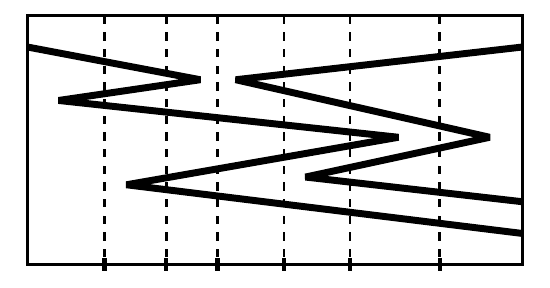}
\end{center}

\caption{Partitioning the graph into small arcs above each of which the separator $M'$ has at most one side wedge.}
\label{fig:partition}
\end{figure}

Observe that since $Z' \cap \pi(T) = \emptyset$ and since $M'$ irreducibly separates $G \times \{0\}$ from $G \times \{1\}$, for each point $a \in Z'$, the set $M' \cap \{a\}_\star$ contains an odd number of points.  Let $k$ be the maximum cardinality of $M' \cap \{a\}_\star$, among all $a \in Z'$.  Then in particular $k$ is odd.

Fix two adjacent points $a,b \in Z'$.  Let $M' \cap \{a\}_\star = \{(a,y_1),\ldots,(a,y_j)\}$, where $j \leq k$ is odd, and $y_1 < y_2 < \cdots < y_j$.  For each $i = 1,\ldots,j$, let $C_i$ be the component of $M' \cap [a,b]_\star$ containing the point $(a,y_i)$.

If there is no side wedge in $M' \cap [a,b]_\star$, then define $S^{[a,b]}_i = C_i$ for each $i = 1,\ldots,j$.

On the other hand, suppose that $M' \cap [a,b]_\star$ has a component $W$ which has a side wedge.  Assume without loss of generality that $a \in \pi(W)$, so that $b \notin \pi(W)$.  Clearly $W \cap \{a\}_\star$ consists of two consecutive points, say $(a,y_m)$ and $(a,y_{m+1})$ of $M'\cap \{a\}_\star$.  Then $C_m = C_{m+1} = W$.  Observe that $M' \cap \{b\}_\star$ contains exactly $j-2$ points.

For each $i = m+2,\ldots,j$, let $J_i \subset [a,b]$ be a closed subarc such that $J_i \cap \pi(W) = \emptyset$, and $J_{i+1}$ is between $J_i$ and $b$ for each $i = m+2,\ldots,j-1$.  For each $i = m+2,\ldots,j$, in a small neighborhood of $C_i$ in $U$, define three arcs $C_i^1,C_i^2,C_i^3$ such that
\begin{itemize}
\item $\pi$ is one-to-one on $C_i^p$ for each $p = 1,2,3$;
\item $C_i^1$ and $C_i^2$ have a common endpoint, and $C_i^2$ and $C_i^3$ have a common endpoint, but these three arcs are otherwise pairwise disjoint;
\item $C_i^1 \cap \{a\}_\star = C_i \cap \{a\}_\star$ and $C_i^3 \cap \{b\}_\star = C_i \cap \{b\}_\star$; and
\item $\pi(C_i^2) = J_i = \pi(C_i^1) \cap \pi(C_i^3)$.
\end{itemize}
We call this procedure ``adding a zig-zag'' to $C_i$.  Refer to Figure \ref{fig:make stairwell} for an illustration.

\begin{figure}
\begin{center}
\includegraphics{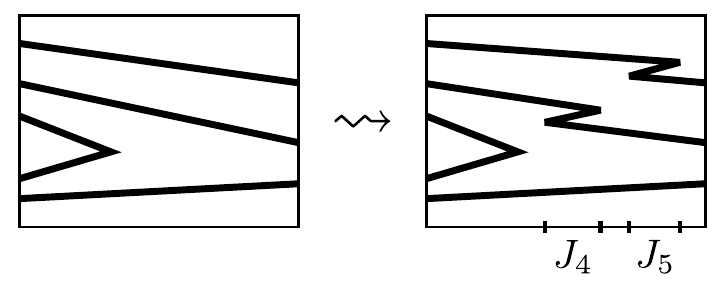}
\end{center}

\caption{Adding ``zig-zags'' to two components of $M' \cap [a,b]_\star$, above the arcs $J_4$ and $J_5$, in order to obtain a set with a stairwell structure.}
\label{fig:make stairwell}
\end{figure}

Now for each $i = 1,\ldots,m-1$, define $S^{[a,b]}_i = C_i$.  For $i = m,\ldots,k$, we define $S^{[a,b]}_i$ by defining the components of these sets in steps, as follows.

Decompose the side wedge $W$ into two arcs $W_m$ and $W_{m+1}$, where $\pi$ is one-to-one on each of $W_m$ and $W_{m+1}$, and $W_i$ contains $(a,y_i)$ for both $i = m,m+1$.  We start with $S^{[a,b]}_i = W_i$ for both $i = m,m+1$.  Then for each $i = m+2,\ldots,j$, in order, we start with $S^{[a,b]}_i = C_i^1$, and we add $C_i^2$ to $S^{[a,b]}_{i-1}$ and add $C_i^3$ to $S^{[a,b]}_{i-2}$.  Finally, for $i = j+1,\ldots,k$, let $S^{[a,b]}_i = \emptyset$.

Define $S_i = \bigcup \{S^{[a,b]}_i: a,b \in Z'$ are adjacent$\}$ for each $i = 1,\ldots,k$, and let $S = \bigcup_{i=1}^k S_i$.  Observe that $S$ is in $U$, and clearly $S$ irreducibly separates $G \times \{0\}$ from $G \times \{1\}$ because $M'$ does.

It is clear that $S^{[a,b]}_i$ is straight for each adjacent pair $a,b \in Z'$ and each $i = 1,\ldots,k$.  Moreover, if $a \in Z'$ and $S \cap \{a\}_\star = \{(a,y_1),\ldots,(a,y_j)\}$, where $j \leq k$ and $y_1 < \cdots < y_j$, then from the construction we see that $S^{[a,x]}_i \cap \{a\}_\star = \{(a,y_i)\}$ for any $x \in Z'$ adjacent to $a$ and each $i = 1,\ldots,j$ (and $S^{[a,x]}_i \cap \{a\}_\star = \emptyset$ if $i > j$).  It follows that $S_i$ is straight for each $i = 1,\ldots,k$.  Thus property \ref{enum:straight} holds.

Let $\alpha_1 = \beta_k = \emptyset$, and for each $i = 1,\ldots,k-1$, let $\beta_i = \alpha_{i+1} = S_i \cap S_{i+1}$.  The points of these sets $S_i \cap S_{i+1}$ are exactly the tips of side wedges and the zig-zag turning points.  Clearly all such points belong to the end sets of the sets $S_i$, and there are no other points in the end sets of the $S_i$'s because $S$ irreducibly separates $G \times \{0\}$ from $G \times \{1\}$ in $G_\star$.  Thus property \ref{enum:ends decomp} holds.

Properties \ref{enum:same side} and \ref{enum:generic} are immediate from the construction.

For property \ref{enum:consistent}, let $C$ be a component of $G \smallsetminus \pi(S_i)$ for some $i \in \{1,\ldots,k\}$.  Note that if $z\in C\cap Z'$, then $|M'\cap \{z\}_{\ast}|<i$.  If $C \subset [a,b]$ for some adjacent pair $a,b \in Z'$, then it is clear from the construction (refer to the right side of Figure \ref{fig:make stairwell}) that $\partial C \subset \pi(\alpha_i)$ or $\partial C \subset \pi(\beta_i)$.  Suppose, on the other hand, that $x_1,x_2 \in \partial C$ do not belong to the same component of $G \smallsetminus Z'$.  Let $a_1,\ldots,a_n \in Z'$ such that $a_p$ and $a_{p+1}$ are adjacent for each $p = 1,\ldots,n-1$, $x_1 \in [a_1,a_2]$, $x_2 \in [a_{n-1},a_n]$, and $[a_p,a_{p+1}] \subset C$ for all $p = 2,\ldots,n-2$.  For each $p = 1,\ldots,n$, let $j_p$ be the number of points in $S \cap \{a_p\}_\star$.  Then $j_p$ is odd for all $p = 1,\ldots,n$.
Since $a_i\in C$ for $p=2,\dots,n-1$, $j_p<i$ for $p=2,\dots,n-1$ and,  since $|j_p - j_{p+1}| = 0$ or $2$ for each $p = 1,\ldots,n-1$, $i\le j_1\le i+1$ and $i\le j_n\le i+1$. Moreover,
since $|j_p - j_{p+1}| = 0$ or $2$ for each $p = 1,\ldots,n-1$, we must have that $j_1$ and $j_n$ have the same parity and, hence, $j_1 = j_n$.
It is now easy to see that each of $x_1$ and $x_2$ corresponds to the tip point of a side wedge or a turning point of a zig-zag joining $S_{j_1-1}$ and $S_{j_1}$. Hence, if $i=j_1=j_n$, $\{x_1,x_2\}\subset \pi(\alpha_i)$ and, if $i=j_1-1=j_n-1$, then $\{x_1,x_2\}\subset \pi(\beta_i)$ and it follows that
\ref{enum:consistent} holds.

Since \ref{enum:straight}--\ref{enum:generic} hold  $\langle S_1,\ldots,S_k \rangle$ is a stairwell structure of odd height $k$ for $S$.
\end{proof}

To illustrate that the procedure indicated in Figure \ref{fig:make simple} may indeed be needed, we offer an example in Figure \ref{fig:non simple} of a set $S$ in $G_\star$, where $G$ is a \emph{simple triod} with legs $T_1$, $T_2$, $T_3$; that is, $G$ is the union of three arcs $T_1$, $T_2$, $T_3$ which have one common endpoint and are otherwise pairwise disjoint.  In this case, $G_\star$ is a ``$3$-page book'', whose three ``pages'' are the squares drawn in Figure \ref{fig:non simple}.  The left edges of the three squares are identified.  We leave it to the reader to observe that this set $S$ irreducibly separates $G \times \{0\}$ from $G \times \{1\}$.  The reader may find it informative to remove the unwanted branch point using the procedure indicated in Figure \ref{fig:make simple} (note that there are two essentially different ways to do this), and then to nudge the set so that all the turning points have distinct projections, and add zig-zags as in Figure \ref{fig:make stairwell}, to obtain a set with a stairwell structure.

\begin{figure}
\begin{center}
\includegraphics{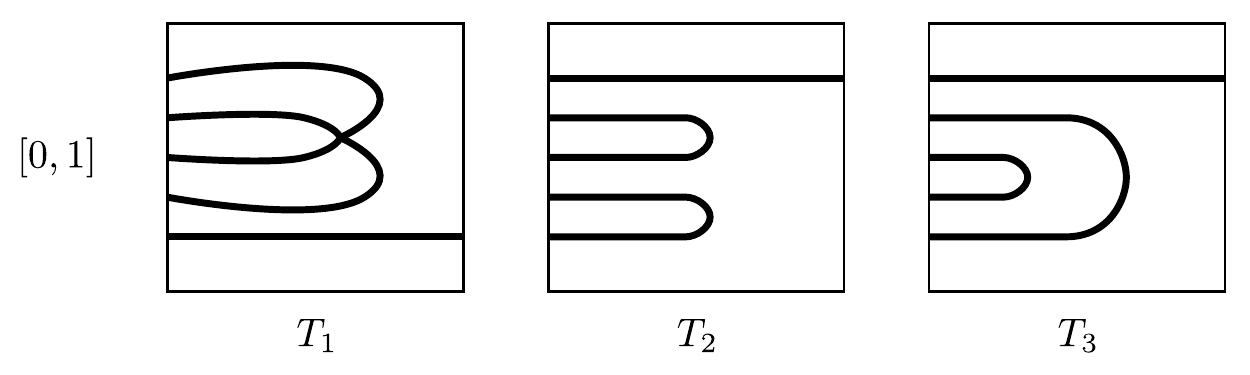}
\end{center}

\caption{An example of an irreducible separator in $G_\star$, for $G$ a simple triod, with an extra unwanted branch point.}
\label{fig:non simple}
\end{figure}

\section{Unfolding stairwells}
\label{sec:unfolding}

In this technical section, we develop the machinery we need to simplify a set with a stairwell structure by taking its inverse image under a simple fold.  As will be seen below, one can reduce the height of a stairwell by taking inverse images under a sequence of simple folds.  In the intermediate stages of this process, the resultant sets will not have a stairwell structure; however, they will exhibit a structure very close to it, which is captured by the next definition of a \emph{broken stairwell structure}.

A broken stairwell structure differs from a stairwell structure in that it contains an additional ``detour'' (which we call a \emph{pit}) at one of the levels.  We will observe in Proposition \ref{prop:reduce height} that a set with a stairwell structure of height $k$ can be relabelled so as to have a broken stairwell structure of height $k - 2$, with a pit at the first level.  We will then prove in Proposition \ref{prop:unfold} that given a broken stairwell structure, we can take the inverse image under a simple fold to obtain a new set with a broken stairwell structure of the same height in which the pit is at the next level up.  This procedure can be repeated to move the pit up to the highest level.  Then, once the pit is at the highest level, applying this procedure once more removes the pit altogether, leaving a set with a stairwell structure (not broken).  This will be carried out formally in the proof of Theorem \ref{thm:lift} in the next section.

\begin{defn}
\label{defn:broken stairwell}
Let $S \subset G_\star$.  A \emph{broken stairwell structure for $S$ of height $k$ with a pit at level $i_0$} is a tuple $\langle S_1,\ldots,S_k; P_1,P_2 \rangle$ such that:
\begin{enumerate}[label=\textbf{\boldmath (S\arabic{*}$'$)}]
\item \label{enum:straight'} $S_1,\ldots,S_k,P_1,P_2$ are non-empty straight subsets of $G_\star$ with $S = S_1 \cup \cdots \cup S_k \cup P_1 \cup P_2$;
\item \label{enum:ends decomp'} Property \ref{enum:ends decomp} above holds for $S_1,\ldots,S_k$, except that $\setends{S_{i_0}}$ is decomposed into three disjoint finite sets: $\setends{S_{i_0}} = \alpha_{i_0} \cup \beta_{i_0} \cup \gamma_{i_0}$.  Additionally, $\setends{P_2} = \setends{P_1} \cup \gamma_{i_0}$, and $\setends{P_1} \cap \gamma_{i_0} = \emptyset$;
\item \label{enum:same side'} Property \ref{enum:same side} above holds for $S_1,\ldots,S_k$, and additionally, there is a neighborhood $V$ of $\pi(\setends{P_1})$ such that $\pi(P_1) \cap V = \pi(P_2) \cap V$, and a neighborhood $W$ of $\pi(\gamma_{i_0})$ such that $\pi(P_2) \cap W = \pi(S_{i_0}) \cap W$;
\item \label{enum:consistent'} Property \ref{enum:consistent} above holds for $S_1,\ldots,S_k$, and additionally, $\pi(S_{i_0})$ has consistent complement relative to $\pi(\gamma_{i_0})$, and $\pi(P_2)$ has consistent complement relative to $\pi(\setends{P_1})$ and to $\pi(\gamma_{i_0})$;
\item \label{enum:generic'} The family $\langle \pi(\alpha_2),\ldots,\pi(\alpha_k),\pi(\setends{P_1}),\pi(\gamma_{i_0}) \rangle$ (which is equal to \\ $\langle \pi(\beta_1),\ldots,\pi(\beta_{k-1}),\pi(\setends{P_1}),\pi(\gamma_{i_0}) \rangle$) is generic in $G$; and
\item \label{enum:pit'} $\pi(\alpha_{i_0}) \cap \pi(P_1 \cup P_2) = \emptyset$.
\end{enumerate}
\end{defn}

See Figure \ref{fig:brokenstairwell} for a simple example of a set with a broken stairwell structure.

Note that even though the sets $S_1,\ldots,S_k,P_1,P_2$ are all non-empty, we do allow for the possibilities that $\alpha_i = \emptyset$ for some values of $i \in \{2,\ldots,k\}$, that $\setends{P_1} = \emptyset$, and that $\gamma_{i_0} = \emptyset$.  See also the remarks immediately following Proposition \ref{prop:reduce height} below.

\begin{figure}
\begin{center}
\includegraphics{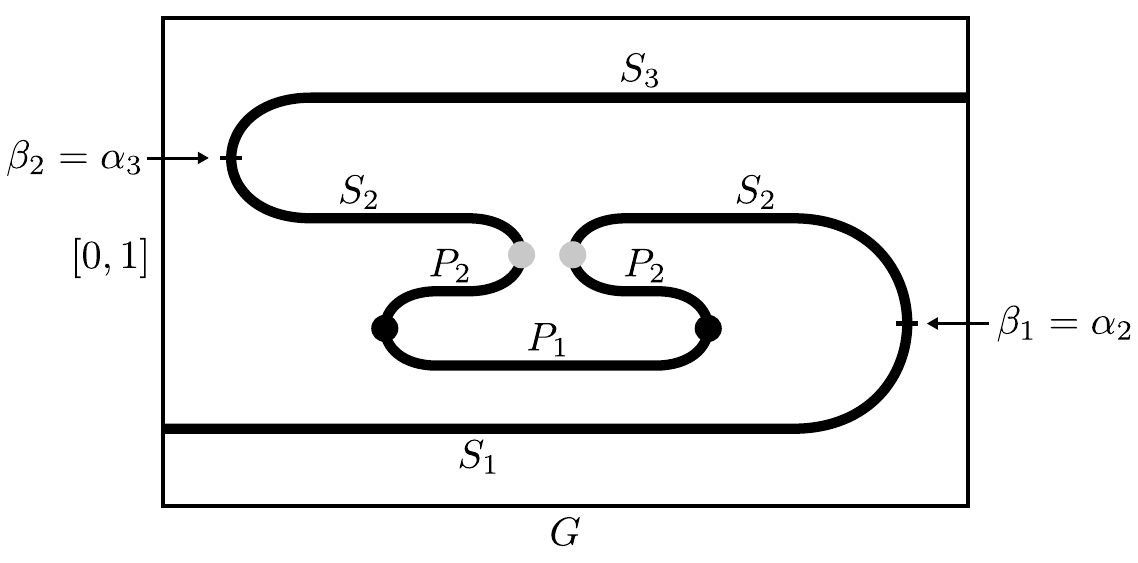}
\end{center}

\caption{An example of a set with a broken stairwell structure of height $3$ with a pit at level $2$ in $G_\star$, where $G$ is an arc.  The points marked with grey dots comprise the set $\gamma_2$, and the points marked with black dots comprise the set $\setends{P_1}$.}
\label{fig:brokenstairwell}
\end{figure}

We make the following observation: if $S \subset G_\star$ is a set with a broken stairwell structure with a pit at level $i_0$, and if $C$ is a component of $G$ such that $S_i \cap C_\star \neq \emptyset$ for each $i = 1,\ldots,k$ and $P_j \cap C_\star \neq \emptyset$ for $j = 1,2$, then $S \cap C_\star$ has a broken stairwell structure with a pit at level $i_0$ obtained by intersecting each of the sets $S_i,P_1,P_2,\alpha_i,\beta_i,$ and $\gamma_{i_0}$ with $C_\star$.

\begin{prop}
\label{prop:reduce height}
Every set $S \subset G_\star$ which has a stairwell structure of height $k$ has a broken stairwell structure of height $k-2$ with a pit at level $1$.
\end{prop}

\begin{proof}
Suppose $S \subset G_\star$ has a stairwell structure $\langle S_1,\ldots,S_k \rangle$ of height $k$.  Let $P_1 = S_1$, $P_2 = S_2$, and for each $i = 1,\ldots,k-2$, let $S_i' = S_{i+2}$.  For each $i = 2,\ldots,k-2$, let $\alpha_i' = \alpha_{i+2}$ and $\beta_i' = \beta_{i+2}$.  Let $\alpha_1' = \emptyset$, $\beta_1' = \beta_3$, and $\gamma_1' = \alpha_3$.

It is now easy to verify that $\langle S_1',\ldots,S_{k-2}'; P_1,P_2 \rangle$ is a broken stairwell structure for $S$ of height $k-2$ with a pit at level $1$.
\end{proof}

We remark that though it may appear at a glance that we could equally well make the pit at level $k-2$ in the above proposition instead of at level $1$, property \ref{enum:pit'} prevents us from doing so in general.

If $\langle S_1,\ldots,S_k; P_1,P_2 \rangle$ is a broken stairwell structure for $S \subset G_\star$ of height $k$ with a pit at level $i_0$, and if $\gamma_{i_0} = \emptyset$, then in fact $\langle S_1,\ldots,S_k \rangle$ is a stairwell structure for $S' = S_1 \cup \cdots \cup S_k \subseteq S$.  Along the same lines, if $\setends{P_1} = \emptyset$ and $G$ is connected, then $\pi(P_1) = G$, and so $\langle P_1 \rangle$ is itself a stairwell structure of height $1$ for $P_1 \subset S$.  For these reasons, we will assume in Proposition \ref{prop:unfold} below that we start with a broken stairwell structure in which $\gamma_{i_0} \neq \emptyset$ and $\setends{P_1} \neq \emptyset$.

Our next major task is to prove Proposition \ref{prop:unfold}.  Because this is a crucial and delicate part at the heart of the results of this paper, we will treat all the details meticulously.  We begin with a lemma to break up and simplify the somewhat involved and tedious proof.

\begin{lem}
\label{lem:straight preimage}
Let $F = F_1 \cup F_2 \cup F_3$ be a simple fold on a graph $G$ with projection $\varphi: F \to G$, and let $S \subset G_\star$ be straight.  Suppose that either $\partial \pi(S) \cap \partial \varphi(F_2) = \emptyset$, or there is a neighborhood $V$ of $\partial \pi(S) \cap \partial \varphi(F_2)$ in $G$ such that $\varphi(F_2) \cap V \subseteq \pi(S) \cap V$.  Then
\begin{enumerate}
\item $S' = \varphi_\star^{-1}(S)$ is straight, and $\setends{S'} = \varphi_\star^{-1}(\setends{S}) \smallsetminus (\partial F_2)_\star$; and
\item $S'' = \varphi_\star^{-1}(S) \cap (F_1 \cup F_2)_\star$ is straight, and
\[ \setends{S''} = \Bigl( \bigl[ \varphi_\star^{-1}(\setends{S}) \cap (F_1 \cup F_2)_\star \bigr] \smallsetminus (\partial F_1)_\star \Bigr) \cup \Bigl( S'' \cap (\partial F_3)_\star \Bigr) .\]
\end{enumerate}
\end{lem}

Refer to Figure \ref{fig:straight preimage} for an illustration of the situation described in Lemma \ref{lem:straight preimage}.

\begin{figure}
\begin{center}
\includegraphics{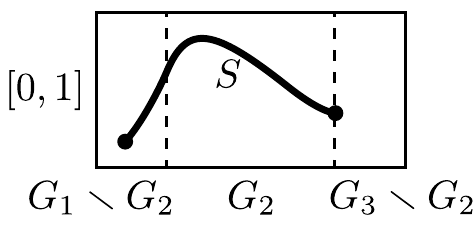} \\
\vspace{0.1in}
\includegraphics{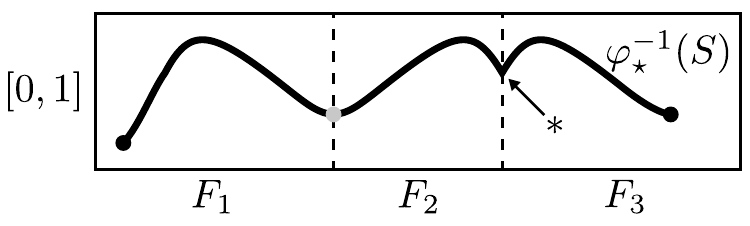}
\end{center}

\caption{On top, a straight set $S$ in $G_\star$ with end set $\setends{S}$ marked with black dots.  Underneath, the preimage of $S$ under the map $\varphi_\star$, with the set $\varphi_\star^{-1}(\setends{S})$ marked with dots, and only those points in $\setends{\varphi_\star^{-1}(S)}$ are in black.  The point marked with $\ast$ belongs to the end set of $\varphi_\star^{-1}(S) \cap (F_1 \cup F_2)_\star$, even though it does not belong to $\varphi_\star^{-1}(\setends{S})$.}
\label{fig:straight preimage}
\end{figure}

\begin{proof}
First, we claim that if $x \in \varphi_\star^{-1}(S) \cap (\partial F_2)_\star$, then there is a neighborhood $W$ of $\pi(x)$ such that $\varphi(W) \subset \pi(S)$.  To see this, we may assume $x \in \varphi_\star^{-1}(S) \cap (\partial F_1)_\star$.  By hypothesis, there is a neighborhood $V$ of $\varphi(\pi(x))$ such that $\varphi(F_2) \cap V \subseteq \pi(S) \cap V$.  We may assume $V$ is small enough so that if we let $W = \varphi^{-1}(V) \cap (F_1 \cup F_2)$, then $W$ is a neighborhood of $\pi(x)$ and $\varphi(F_1 \cap W) = \varphi(F_2 \cap W) = \varphi(F_2) \cap V$.  It follows that $\varphi(W) \subset \pi(S)$.  The argument is similar for $x \in \varphi_\star^{-1}(S) \cap (\partial F_3)_\star$.

For (1), note that clearly $S'$ is closed and $\pi$ is one-to-one on $S'$, since $S$ is closed and $\pi$ is one-to-one on $S$.  For $x \in S' \smallsetminus (\partial F_2)_\star$, $\varphi$ is one-to-one in a neighborhood of $\pi(x)$, and so the component of $x$ in $S'$ is non-degenerate since the component of $\varphi_\star(x)$ in $S$ is non-degenerate.  For $x \in S' \cap (\partial F_2)_\star$, the component of $x$ in $S'$ is non-degenerate by the above claim.  Thus $S'$ is straight.

It is straightforward to see that for $x \notin (\partial F_2)_\star$, we have $x \in \setends{S'}$ if and only if $\varphi_\star(x) \in \setends{S}$, since $\varphi$ is one-to-one on a neighborhood of $\pi(x)$.  Moreover, by the above claim, clearly $\setends{S'} \cap (\partial F_2)_\star = \emptyset$.  This establishes (1).

For (2), it can be argued similarly that $S''$ is a straight.  As for the end set of $S''$, clearly $\setends{S''} \subset (F_1 \cap F_2)_\star$ since $S'' \subset (F_1 \cap F_2)_\star$.  As in (1), it is straightforward to see that for $x \in S'' \smallsetminus (\partial F_2)_\star$, we have $x \in \setends{S''}$ if and only if $\varphi_\star(x) \in \setends{S}$, and by the claim, $\setends{S''} \cap (\partial F_1)_\star = \emptyset$.  Finally, if $x \in S'' \cap (\partial F_3)_\star$, then clearly any neighborhood of $\pi(x)$ meets both $\pi(S'')$ and the complement of $\pi(S'')$ (since it meets the interior of $F_3$), therefore $x \in \setends{S''}$.  This establishes (2).
\end{proof}

\begin{prop}
\label{prop:unfold}
Let $G$ be a connected graph, and let $S \subset G_\star$ have a broken stairwell structure $\langle S_1,\ldots,S_k; P_1,P_2 \rangle$ of height $k$ with a pit at level $i_0 \leq k$, in which $\gamma_{i_0} \neq \emptyset$ and $\setends{P_1} \neq \emptyset$.  Then there exists a simple fold $\varphi: F \to G$ such that $F$ is connected, and $\varphi_\star^{-1}(S)$ contains a set $S'$ with a broken stairwell structure of height $k$ with a pit at level $i_0 + 1$ if $i_0 < k$, or simply a stairwell structure of height $k$ if $i_0 = k$.
\end{prop}

\begin{proof}
Recall from the comment immediately following Lemma \ref{lem:fold basic} that to uniquely define a simple fold, it suffices to choose three subsets $G_1,G_2,G_3$ of $G$ satisfying properties \ref{enum:regular}, \ref{enum:G1 G3}, and \ref{enum:separate}.  We will define a simple fold in this way, relying on Proposition \ref{prop:define fold} to verify these properties.

Define the simple fold $F = F_1 \cup F_2 \cup F_3$ by $F_1 \approx G_1 = \pi(P_1)$, $F_2 \approx G_2 = \pi(P_2)$, and $F_3 \approx G_3 = \sigma_{\pi(\gamma_{i_0})}(\pi(S_{i_0}))$, and let $\varphi: F \to G$ be the projection.  Hence, $F$ is the union of $F_1$, $F_2$ and $F_3$ with $F_1$ glued to $F_2$ along the part corresponding to  $\setends{P_1}$ and $F_2$ is glued to $F_3$ along the part corresponding to $\pi(\gamma_{i_0})$. Note that since $G$ is connected, we have by the remarks following Definition \ref{defn:cons comp} and by Proposition \ref{prop:side eq} and \ref{enum:same side'} for $S$ that $\pi(P_1) = \sigma_{\pi(\setends{P_1})}(\pi(P_2))$ and $\sigma_{\pi(\gamma_{i_0})}(\pi(S_{i_0})) = \sigma_{\pi(\gamma_{i_0})}(\pi(P_2))$.  So by \ref{enum:ends decomp'} and Proposition \ref{prop:define fold}, these three sets $G_1,G_2,G_3$ do indeed define a simple fold.

We record the following basic observations for reference below:
\begin{enumerate}[resume, label=(\thethm.\arabic{*})\hspace{0.1in}, ref=(\thethm.\arabic{*})]
\item \label{obs:1} $\partial \varphi(F_1) = \varphi(\partial F_1) = \pi(\setends{P_1})$
\item \label{obs:2} $\partial \varphi(F_3) = \varphi(\partial F_3) = \pi(\gamma_{i_0})$
\item \label{obs:3} $\partial \varphi(F_2) = \partial \varphi(F_1) \cup \partial \varphi(F_3) = \pi(\setends{P_1}) \cup \pi(\gamma_{i_0})$
\end{enumerate}

We now describe the set $S' \subseteq \varphi_\star^{-1}(S)$ and its (broken) stairwell structure piece by piece.  The reader will find it helpful to refer to Figure \ref{fig:unfold} when reading the following definitions.

For each $i \notin \{i_0,i_0+1\}$, define
\begin{align*}
S_i' &= \varphi_\star^{-1}(S_i) \\
\alpha_i' &= \varphi_\star^{-1}(\alpha_i) \\
\beta_i' &= \varphi_\star^{-1}(\beta_i) .
\intertext{\indent For level $i_0$, define}
S_{i_0}' &= \bigl[ \varphi_\star^{-1}(P_1) \cap (F_1)_\star \bigr] \cup \bigl[ \varphi_\star^{-1}(P_2) \cap (F_2)_\star \bigr] \cup \bigl[ \varphi_\star^{-1}(S_{i_0}) \cap (F_3)_\star \bigr] \\
\alpha_{i_0}' &= \varphi_\star^{-1}(\alpha_{i_0}) \\
\beta_{i_0}' &= \varphi_\star^{-1}(\beta_{i_0}) \cap (F_3)_\star .
\intertext{\indent If $i_0 < k$, then further define}
S_{i_0+1}' &= \varphi_\star^{-1}(S_{i_0+1}) \\
\alpha_{i_0+1}' &= \varphi_\star^{-1}(\alpha_{i_0+1}) \cap (F_3)_\star \\
\beta_{i_0+1}' &= \varphi_\star^{-1}(\beta_{i_0+1}) \\
\gamma_{i_0+1}' &= \varphi_\star^{-1}(\alpha_{i_0+1}) \cap (F_1 \cup F_2)_\star ,
\intertext{as well as}
P_1' &= \varphi_\star^{-1}(P_2) \cap (F_1 \cup F_2)_\star \\
P_2' &= \varphi_\star^{-1}(S_{i_0}) \cap (F_1 \cup F_2)_\star .
\end{align*}

\begin{figure}
\begin{center}
\includegraphics{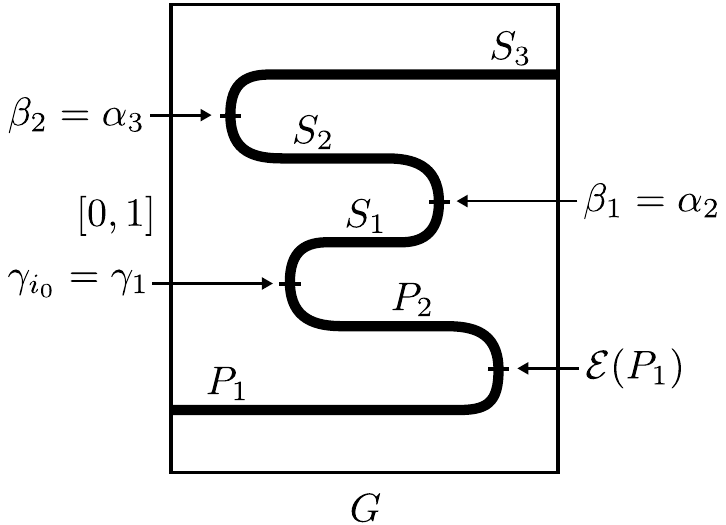} \\
\vspace{0.1in}
\includegraphics{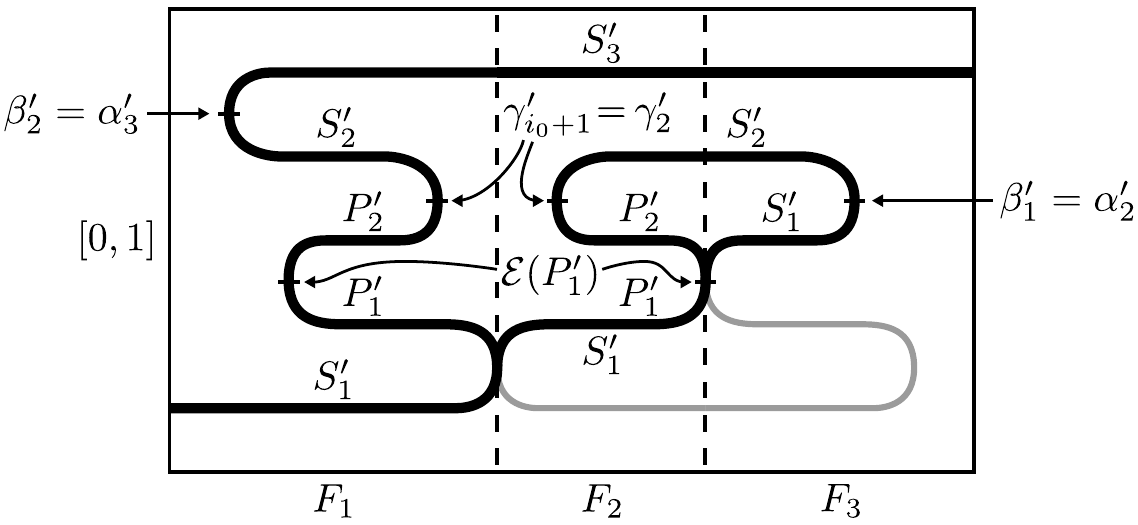}
\end{center}

\caption{On top, a set $S$ with a broken stairwell structure of height $3$ with a pit at level $1$.  Underneath, the preimage of $S$ under the map $\varphi_\star$, with the subset $S'$ with a broken stairwell structure of height $3$ with a pit at level $2$ in black.  Note that $S_1'$ and $P_1'$ overlap in a segment in $(F_2)_\star$.}
\label{fig:unfold}
\end{figure}

We now proceed with confirming that the above sets comprise a (broken) stairwell structure.  We begin by showing that the sets $S_1',\ldots,S_k',P_1',P_2'$ are all straight, and computing their end sets.

\medskip \noindent
\textbf{Straightness and end sets.}
\nopagebreak

For $i \neq i_0$, we have by \ref{enum:ends decomp'} and \ref{enum:generic'} for $S$ that $\partial \pi(S_i) = \pi(\alpha_i) \cup \pi(\beta_i)$ is disjoint from $\partial \varphi(F_2) = \pi(\setends{P_1}) \cup \pi(\gamma_{i_0})$ (by \ref{obs:3}).  Therefore, by Lemma \ref{lem:straight preimage}, $S_i' = \varphi_\star^{-1}(S_i)$ is straight, and
\begin{enumerate}[resume, label=(\thethm.\arabic{*})\hspace{0.1in}, ref=(\thethm.\arabic{*})]
\item \label{obs:4} $\setends{S_i'} = \varphi_\star^{-1}(\setends{S_i})$ for $i \neq i_0$.
\end{enumerate}
Observe that $\varphi_\star^{-1}(\setends{S_i})$ is disjoint from $(\partial F_2)_\star$.

We now consider $S_{i_0}'$.  Because $\varphi(F_1) = \pi(P_1)$, $\varphi(F_2) = \pi(P_2)$, and $\varphi(F_3) = \sigma_{\pi(\gamma_{i_0})}(\pi(S_{i_0})) \supseteq \pi(S_{i_0})$, clearly each of $\varphi_\star^{-1}(P_1) \cap (F_1)_\star$, $\varphi_\star^{-1}(P_2) \cap (F_2)_\star$, and $\varphi_\star^{-1}(S_{i_0}) \cap (F_3)_\star$ is straight, as $P_1$, $P_2$, and $S_{i_0}$ are straight.  From the equalities
\[ \varphi_\star^{-1}(P_1) \cap (\partial F_1)_\star = \varphi_\star^{-1}(\setends{P_1}) \cap (\partial F_1)_\star = \varphi_\star^{-1}(P_2) \cap (\partial F_1)_\star \]
and
\[ \varphi_\star^{-1}(P_2) \cap (\partial F_3)_\star = \varphi_\star^{-1}(\gamma_{i_0}) \cap (\partial F_3)_\star = \varphi_\star^{-1}(S_{i_0}) \cap (\partial F_3)_\star \]
it follows that $\pi$ is one-to-one on $S_{i_0}'$.  Thus $S_{i_0}'$ is straight.

For the end set of $S_{i_0}'$, observe that since $\pi(S_{i_0}') \supset F_1 \cup F_2$, $\setends{S_{i_0}'} \subset (F_3)_\star$.  Moreover, $\setends{S_{i_0}'} \cap (\partial F_3)_\star = \emptyset$, because $\pi(S_{i_0})$ agrees with $\varphi(F_3) = \sigma_{\pi(\gamma_{i_0})}(\pi(S_{i_0}))$ near $\pi(\gamma_{i_0})$ in $G$.  Thus
\begin{enumerate}[resume, label=(\thethm.\arabic{*})\hspace{0.1in}, ref=(\thethm.\arabic{*})]
\item \label{obs:5} $F_1 \cup F_2 \subset \intr(\pi(S_{i_0}'))$, and $\setends{S_{i_0}'} \subset \intr(F_3)_\star$.
\end{enumerate}
By the definition of $S_{i_0}'$, we have $S_{i_0}' \cap (F_3)_\star = \varphi^{-1}(S_{i_0}) \cap (F_3)_\star$, and it follows that
\begin{enumerate}[resume, label=(\thethm.\arabic{*})\hspace{0.1in}, ref=(\thethm.\arabic{*})]
\item \label{obs:6} $\setends{S_{i_0}'} = \varphi_\star^{-1}(\setends{S_{i_0}}) \cap \intr(F_3)_\star$.
\end{enumerate}

Looking at both cases ($i \neq i_0$ and $i = i_0$) above, we see that:
\begin{enumerate}[resume, label=(\thethm.\arabic{*})\hspace{0.1in}, ref=(\thethm.\arabic{*})]
\item \label{obs:7} $\setends{S_i'} \cap (\partial F_2)_\star = \emptyset$ for each $i = 1,\ldots,k$.
\end{enumerate}

Next, we consider $P_1' = \varphi_\star^{-1}(P_2) \cap (F_1 \cup F_2)_\star$.  Observe that $\varphi(F_2) = \pi(P_2)$, so Lemma \ref{lem:straight preimage} applies, and we conclude that $P_1'$ is straight.  For the end set of $P_1'$, we have by Lemma \ref{lem:straight preimage} that
\[ \setends{P_1'} = \Bigl( \bigl[ \varphi_\star^{-1}(\setends{P_2}) \cap (F_1 \cup F_2)_\star \bigr] \smallsetminus (\partial F_1)_\star \Bigr) \cup \Bigl( P_1' \cap (\partial F_3)_\star \Bigr) .\]
We simplify this expression using the following straightforward observations:
\begin{itemize}
\item $\varphi^{-1}(\pi(\setends{P_1})) \cap (F_1 \cup F_2) = \partial F_1$, so we can replace $\setends{P_2} = \setends{P_1} \cup \gamma_{i_0}$ by $\gamma_{i_0}$ in the above expression;
\item $\varphi^{-1}(\pi(\gamma_{i_0})) \subset F_1 \cup F_2$ and $\varphi^{-1}(\pi(\gamma_{i_0})) \cap \partial F_1 = \emptyset$ (by \ref{obs:1} and \ref{enum:generic'} for $S$), so $\bigl[ \varphi_\star^{-1}(\gamma_{i_0}) \cap (F_1 \cup F_2)_\star \bigr] \smallsetminus (\partial F_1)_\star = \varphi_\star^{-1}(\gamma_{i_0})$; and
\item $P_1' \cap (\partial F_3)_\star = \varphi_\star^{-1}(\gamma_{i_0})$ by \ref{obs:2}, so $\varphi_\star^{-1}(\gamma_{i_0}) \cup \Bigl( P_1' \cap (\partial F_3)_\star \Bigr) = \varphi_\star^{-1}(\gamma_{i_0})$.
\end{itemize}
We thus have
\begin{enumerate}[resume, label=(\thethm.\arabic{*})\hspace{0.1in}, ref=(\thethm.\arabic{*})]
\item \label{obs:8} $\setends{P_1'} = \varphi_\star^{-1}(\gamma_{i_0})$, which is contained in $(F_1 \cup F_2)_\star$.
\end{enumerate}

Lastly, we consider $P_2' = \varphi_\star^{-1}(S_{i_0}) \cap (F_1 \cup F_2)_\star$.  By \ref{enum:generic'} for $S$, we have that $\partial \varphi(F_1) \cap \partial \pi(S_{i_0}) = \pi(\setends{P_1}) \cap \partial \pi(S_{i_0}) = \emptyset$, which means by \ref{obs:3} that $\partial \varphi(F_2) \cap \partial \pi(S_{i_0}) = \partial \varphi(F_3) = \pi(\gamma_{i_0})$.  By \ref{enum:same side'} for $S$, the sets $\pi(S_{i_0})$ and $\varphi(F_2) = \pi(P_2)$ agree in a neighborhood of $\pi(\gamma_{i_0})$, hence Lemma \ref{lem:straight preimage} applies, and we have that $P_2'$ is straight.

For the end set of $P_2'$, we have by Lemma \ref{lem:straight preimage} that
\[ \setends{P_2'} = \Bigl( \bigl[ \varphi_\star^{-1}(\setends{S_{i_0}}) \cap (F_1 \cup F_2)_\star \bigr] \smallsetminus (\partial F_1)_\star \Bigr) \cup \Bigl( P_2' \cap (\partial F_3)_\star \Bigr) .\]
We simplify this expression using the following straightforward observations:
\begin{itemize}
\item $\partial F_3 \subset \varphi^{-1}(\pi(\gamma_{i_0})) \subset F_1 \cup F_2$ and $\varphi^{-1}(\pi(\gamma_{i_0})) \cap \partial F_1 = \emptyset$ (as above), so since $\setends{S_{i_0}} = \alpha_{i_0} \cup \beta_{i_0} \cup \gamma_{i_0}$, we obtain
\[ \setends{P_2'} = \Bigl( \bigl[ \varphi_\star^{-1}(\alpha_{i_0} \cup \beta_{i_0}) \cap (F_1 \cup F_2)_\star \bigr] \smallsetminus (\partial F_1)_\star \Bigr) \cup \varphi_\star^{-1}(\gamma_{i_0}) ;\]
\item $\varphi^{-1}(\pi(\alpha_{i_0})) \cap (F_1 \cup F_2) = \emptyset$ by \ref{enum:pit'} for $S$, so we can replace $\alpha_{i_0} \cup \beta_{i_0}$ with $\beta_{i_0}$ in the above expression; and
\item $\varphi^{-1}(\pi(\beta_{i_0})) \cap \partial F_1 = \emptyset$, so
\begin{align*}
\bigl[ \varphi_\star^{-1}(\beta_{i_0}) \cap (F_1 \cup F_2)_\star \bigr] \smallsetminus (\partial F_1)_\star &= \varphi_\star^{-1}(\beta_{i_0}) \cap (F_1 \cup F_2)_\star \\
&= \varphi_\star^{-1}(\alpha_{i_0+1}) \cap (F_1 \cup F_2)_\star \\
&= \gamma_{i_0+1}' .
\end{align*}
\end{itemize}
We thus have by \ref{obs:8} that
\begin{enumerate}[resume, label=(\thethm.\arabic{*})\hspace{0.1in}, ref=(\thethm.\arabic{*})]
\item \label{obs:9} $\setends{P_2'} = \setends{P_1'} \cup \gamma_{i_0+1}'$.
\end{enumerate}

\bigskip
We now continue with the remaining properties to show that the above sets comprise a (broken) stairwell structure.

\medskip \noindent
\textbf{\ref{enum:ends decomp} / \ref{enum:ends decomp'}.}
\nopagebreak

For $i \notin \{i_0,i_0+1\}$, we have $\setends{S_i'} = \varphi_\star^{-1}(\alpha_i) \cup \varphi_\star^{-1}(\beta_i) = \alpha_i' \cup \beta_i'$ by \ref{enum:ends decomp'} for $S$, and clearly $\alpha_i' \cap \beta_i' = \emptyset$ since $\alpha_i \cap \beta_i = \emptyset$.  Similarly, if $i_0 < k$, then by \ref{obs:4},
\begin{align*}
\setends{S_{i_0+1}'} &= \varphi_\star^{-1}(\setends{S_{i_0+1}}) \\
&= \varphi_\star^{-1}(\alpha_{i_0+1}) \cup \varphi_\star^{-1}(\beta_{i_0+1}) \\
&= \bigl[ \varphi_\star^{-1}(\alpha_{i_0+1}) \cap (F_1 \cup F_2)_\star \bigr] \cup \bigl[ \varphi_\star^{-1}(\alpha_{i_0+1}) \cap (F_3)_\star \bigr] \cup \varphi_\star^{-1}(\beta_{i_0+1}) \\
&= \gamma_{i_0+1}' \cup \alpha_{i_0+1}' \cup \beta_{i_0+1}' .
\end{align*}
We claim that the sets $\alpha_{i_0+1}',\beta_{i_0+1}',\gamma_{i_0+1}'$ are pairwise disjoint.  Indeed, because $\alpha_{i_0+1} \cap \beta_{i_0+1} = \emptyset$, we immediately have from the definitions of the sets $\alpha_{i_0+1}',\beta_{i_0+1}',\gamma_{i_0+1}'$ that $\alpha_{i_0+1}' \cap \beta_{i_0+1}' = \emptyset = \beta_{i_0+1}' \cap \gamma_{i_0+1}'$.  Moreover, also from these definitions we see that $\alpha_{i_0+1}' \cap \gamma_{i_0+1}' \subseteq (F_3)_\star \cap (F_1 \cup F_2)_\star = (\partial F_3)_\star \subseteq (\partial F_2)_\star$.  But also $\alpha_{i_0+1}',\gamma_{i_0+1}' \subseteq \setends{S_{i_0+1}'}$, and $\setends{S_{i_0+1}'} \cap (\partial F_2)_\star = \emptyset$ by \ref{obs:7}.  Thus $\alpha_{i_0+1}' \cap \gamma_{i_0+1}' = \emptyset$.

For $S_{i_0}'$, we have by \ref{obs:6} and the fact that $\varphi_\star^{-1}(\gamma_{i_0}) \cap \intr(F_3)_\star = \emptyset$ that
\begin{align*}
\setends{S_{i_0}'} &= \varphi_\star^{-1}(\setends{S_{i_0}}) \cap \intr(F_3)_\star \\
&= \bigl[ \varphi_\star^{-1}(\alpha_{i_0}) \cap (F_3)_\star \bigr] \cup \bigl[ \varphi_\star^{-1}(\beta_{i_0}) \cap (F_3)_\star \bigr] .
\end{align*}

Moreover, by \ref{enum:pit'} for $S$ and since $\varphi(F_1 \cup F_2) = \pi(P_1 \cup P_2)$, we have that $\varphi_\star^{-1}(\alpha_{i_0}) \subset (F_3)_\star$, so that $\varphi_\star^{-1}(\alpha_{i_0}) \cap (F_3)_\star = \varphi^{-1}(\alpha_{i_0}) = \alpha_{i_0}'$.  Thus $\setends{S_{i_0}'} = \alpha_{i_0}' \cup \beta_{i_0}'$.  Again, clearly $\alpha_{i_0}' \cap \beta_{i_0}' = \emptyset$ since $\alpha_{i_0} \cap \beta_{i_0} = \emptyset$.

It is straightforward to see that $\beta_i' = \alpha_{i+1}'$ for each $i = 1,\ldots,k-1$, since $\beta_i = \alpha_{i+1}$ for each $i = 1,\ldots,k-1$ by \ref{enum:ends decomp'} for $S$.  The only standout case is when $i = i_0$ (if $i_0 < k$), and here $\beta_{i_0}' = \varphi_\star^{-1}(\beta_{i_0}) \cap (F_3)_\star = \varphi_\star^{-1}(\alpha_{i_0+1}) \cap (F_3)_\star = \alpha_{i_0+1}'$.  Obviously, $\alpha_1' = \emptyset = \beta_k'$ since $\alpha_1 = \emptyset = \beta_k$.

We have already deduced in \ref{obs:9} that $\setends{P_2'} = \setends{P_1'} \cup \gamma_{i_0+1}'$, and the sets $\setends{P_1'} = \varphi_\star^{-1}(\gamma_{i_0})$ and $\gamma_{i_0+1}' = \varphi_\star^{-1}(\alpha_{i_0+1}) \cap (F_1 \cup F_2)_\star$ are clearly disjoint since by \ref{enum:generic'} for $S$, $\gamma_{i_0} \cap \alpha_{i_0+1} = \emptyset$.

\medskip \noindent
\textbf{\ref{enum:same side} / \ref{enum:same side'}.}
\nopagebreak

Because $\pi(\setends{S_i')) \cap \partial F_2 = \emptyset$ for each $i = 1,\ldots,k$ (by \ref{obs:7}), we have that $\varphi$ is one-to-one in a neighborhood of each point of $\pi(\setends{S_i'})$.  It is then straightforward to see from property \ref{enum:same side'} for $S$ and from the definition of $S_i'$ that there is a neighborhood $V$ of $\pi(\beta_i') = \pi(\alpha_{i+1}')$ such that $\pi(S_i'} \cap V = \pi(S_{i+1}') \cap V$.  Again, the only standout case is when $i = i_0$, and here $\pi(\setends{S_{i_0}'}) \subset \intr(F_3)$, and $\pi(S_{i_0}') \cap F_3 = \varphi^{-1}(\pi(S_{i_0})) \cap F_3$, so the neighborhood of $\beta_{i_0} = \alpha_{i_0+1}$ in $G$ in which $\pi(S_{i_0})$ and $\pi(S_{i_0+1})$ agree pulls back under $(\varphi {\upharpoonright}_{F_3})^{-1}$ to a neighborhood of $\beta_{i_0}' = \alpha_{i_0+1}'$ in which $\pi(S_{i_0}')$ and $\pi(S_{i_0+1}')$ agree.

If $i_0 < k$, then by \ref{obs:7}, we in particular have that $\pi(\gamma_{i_0+1}') \cap \partial F_2 = \emptyset$, and so $\varphi$ is one-to-one in a neighborhood of each point of $\pi(\gamma_{i_0+1}')$.  Then as above we have that there is a neighborhood of $\pi(\gamma_{i_0+1}') \subset F_1 \cup F_2$ on which $\pi(S_{i_0+1}') = \varphi^{-1}(\pi(S_{i_0+1}))$ and $P_2' = \varphi^{-1}(\pi(S_{i_0})) \cap (F_1 \cup F_2)$ agree.

For $P_1'$ and $P_2'$, recall from \ref{obs:8} that $\setends{P_1'} = \varphi_\star^{-1}(\gamma_{i_0})$, which is contained in $(F_1 \cup F_2)_\star$.  Let $z \in \pi(\setends{P_1'})$.  Note that $z \notin \partial F_1$ since $\varphi(\partial F_1) = \pi(\setends{P_1})$ by \ref{obs:1}, and $\pi(\setends{P_1}) \cap \pi(\gamma_{i_0}) = \emptyset$ by \ref{enum:generic'} for $S$.  If $z \notin \partial F_3$, then $\varphi$ is one-to-one in a neighborhood of $z$, so as above there is a neighborhood of $z$ on which $\pi(P_1')$ and $\pi(P_2')$ agree.

If $z \in \partial F_3$, then by \ref{enum:generic'} for $S$, $\varphi(z)$ is not a branch point of $G$, so there is a neighborhood of $z$ in $F$ which is homeomorphic to an open arc $J$.  $J \smallsetminus \{z\}$ is the union of two open arcs $J_1$ and $J_2$, where $J_1 \subset \intr(F_1 \cup F_2)$ and $J_2 \subset \intr(F_3)$.  Since $P_1'$ and $P_2'$ are contained in $(F_1 \cup F_2)_\star$ and $z \in \pi(\setends{P_1'}) \subset \pi(\setends{P_2'})$ (by \ref{obs:9}), there is a neighborhood $W \subset J$ of $z$ in $F$ such that $W \cap (F_1 \cup F_2) \subset \pi(P_1') \cap \pi(P_2')$.  On the other hand, $W \cap \intr(F_3)$ is disjoint from $\pi(P_1')$ and from $\pi(P_2')$.  Thus $\pi(P_1') \cap W = \pi(P_2') \cap W$.

\medskip \noindent
\textbf{\ref{enum:consistent} / \ref{enum:consistent'}.}
\nopagebreak

Given $i \notin \{i_0,i_0+1\}$, let $C$ be a component of $F \smallsetminus \pi(S_i') = F \smallsetminus \varphi^{-1}(\pi(S_i))$.  Then $\varphi(C)$ is contained in a component of $G \smallsetminus \pi(S_i)$, hence $\overline{\varphi(C)}$ meets at most one of $\pi(\alpha_i)$ and $\pi(\beta_i)$.  It follows that $\partial C \subseteq \pi(\alpha_i') = \varphi^{-1}(\pi(\alpha_i))$ or $\partial C \subseteq \pi(\beta_i') = \varphi^{-1}(\pi(\beta_i))$.

For level $i_0$, let  $C$ be a component of $F \smallsetminus \pi(S_{i_0}')$.  By \ref{obs:5}, $F_1 \cup F_2 \subset \intr(\pi(S_{i_0}'))$, hence $\overline{C} \subset \intr(F_3)$.  Moreover, $C \subset F_3 \smallsetminus \varphi^{-1}(\pi(S_{i_0}))$ since $S_{i_0}' \cap (F_3)_\star = \varphi^{-1}(S_{i_0}) \cap (F_3)_\star$.  Thus again $\varphi(C)$ is contained in a component of $G \smallsetminus \pi(S_{i_0})$, hence $\overline{\varphi(C)}$ meets at most one of $\pi(\alpha_{i_0})$, $\pi(\beta_{i_0})$, and $\pi(\gamma_{i_0})$.  Note however that $\overline{\varphi(C)} \cap \pi(\gamma_{i_0}) = \emptyset$ since $\varphi^{-1}(\pi(\gamma_{i_0})) \cap F_3 = \partial F_3$ and $\overline{C} \cap \partial F_3 = \emptyset$.  It follows that $\overline{C}$ meets at most one of $\pi(\alpha_{i_0}') = \varphi^{-1}(\pi(\alpha_{i_0}))$ and $\pi(\beta_{i_0}') = \varphi^{-1}(\pi(\beta_{i_0})) \cap F_3$.

Now suppose that $i_0 < k$, and consider level $i_0+1$.  Let $C$ be a component of $F \smallsetminus \pi(S_{i_0+1}')$.  Since $S_{i_0+1}' = \varphi_\star^{-1}(S_{i_0+1})$ and $\beta_{i_0+1}' = \varphi_\star^{-1}(\beta_{i_0+1})$, we have as above that if $\partial C \cap \pi(\beta_{i_0+1}') \neq \emptyset$, then $\partial C \subset \pi(\beta_{i_0+1}')$.

Suppose, on the other hand, that $\partial C \cap \pi(\alpha_{i_0+1}') \neq \emptyset$ or $\partial C \cap \pi(\gamma_{i_0+1}') \neq \emptyset$.  Then $\partial \varphi(C) \cap \pi(\alpha_{i_0+1}) \neq \emptyset$.  It follows that $\varphi(C)$ is contained in a component $\tilde{C}$ of $G \smallsetminus \pi(S_{i_0+1})$ whose boundary is contained in $\pi(\alpha_{i_0+1})$.  By Proposition \ref{prop:side eq}, $\tilde{C}$ is also a component of $G \smallsetminus \pi(S_{i_0})$ whose boundary is contained in $\pi(\beta_{i_0})$, because $\pi(S_{i_0+1})$ and $\pi(S_{i_0})$ agree in a neighborhood of $\pi(\alpha_{i_0+1}) = \pi(\beta_{i_0})$, and $\pi(S_{i_0})$ has consistent complement relative to $\pi(\beta_{i_0})$.

Observe that $\varphi(\partial F_3) = \pi(\gamma_{i_0}) \subset \pi(S_{i_0})$ and $\tilde{C} \subseteq G \smallsetminus \pi(S_{i_0})$, so $C \cap \partial F_3 = \emptyset$.  Because $C$ is connected, by Lemma \ref{lem:fold basic}\ref{enum:inter sep} this means $C \subseteq F_1 \cup F_2$ or $C \subseteq F_3$.  Therefore, by the definitions of $\alpha_{i_0+1}'$ and $\gamma_{i_0+1}'$, either $\partial C \subset \pi(\alpha_{i_0+1}')$ or $\partial C \subset \pi(\gamma_{i_0+1}')$.

Now let $D$ be a component of $F \smallsetminus \pi(P_2')$.  Note that $\partial F_3 \subset \varphi^{-1}(\pi(\gamma_{i_0})) \cap (F_1 \cup F_2) \subset \varphi^{-1}(\pi(S_{i_0})) \cap (F_1 \cup F_2) = \pi(P_2')$, so $D \cap \partial F_3 = \emptyset$.  By Lemma \ref{lem:fold basic}\ref{enum:inter sep}, this means $D \subset F_1 \cup F_2$ or $D \subset F_3$.

If $D \subset F_3$, then since $F_3 \cap \pi(P_2') = \partial F_3$, we must have $\partial D \subset \partial F_3 \subset \pi(\setends{P_1'})$.  If $D \subset F_1 \cup F_2$, then since $\pi(P_2') = \varphi^{-1}(\pi(S_{i_0})) \cap (F_1 \cup F_2)$, we have that $\varphi(D)$ is contained in a component $\tilde{D}$ of $G \smallsetminus \pi(S_{i_0})$.  Moreover, $\varphi(D) \subset \pi(P_1 \cup P_2)$, so $\partial \tilde{D} \cap  \pi(\alpha_{i_0}) = \emptyset$ by \ref{enum:pit'} for $S$.  This means either $\partial \tilde{D} \subset \pi(\beta_{i_0})$ or $\partial \tilde{D} \subset \pi(\gamma_{i_0})$.  Then because $\setends{P_1'} = \varphi_\star^{-1}(\gamma_{i_0})$ and $\gamma_{i_0+1}' = \varphi_\star^{-1}(\alpha_{i_0+1}) \cap (F_1 \cup F_2)_\star = \varphi_\star^{-1}(\beta_{i_0}) \cap (F_1 \cup F_2)_\star$, it follows that $\partial D \subset \pi(\setends{P_1'})$ or $\partial D \subset \pi(\gamma_{i_0+1}')$.

\medskip \noindent
\textbf{\ref{enum:generic} / \ref{enum:generic'}.}
\nopagebreak

Since $\partial \pi(P_2)$ is disjoint from the set $Z$ of branch points and endpoints of $G$, we have that the set of branch points and endpoints of $F$ is $\varphi^{-1}(Z)$.  It is then trivial to see from the definitions of the sets $\alpha_2',\ldots,\alpha_k',\gamma_{i_0+1}',\setends{P_1'}$, and from property \ref{enum:generic'} for $S$, that the family $\langle \pi(\alpha_2'),\ldots,\pi(\alpha_k'),\pi(\gamma_{i_0+1}'),\pi(\setends{P_1'}) \rangle$ (or simply $\langle \pi(\alpha_2'),\ldots,\pi(\alpha_k') \rangle$ in the case $i_0 = k$) is generic.

\medskip \noindent
\textbf{\ref{enum:pit'}.}
\nopagebreak

Recall that $\alpha_{i_0+1}' = \beta_{i_0}' \subset \intr(F_3)_\star$ by \ref{obs:5}, which means that $\alpha_{i_0+1}' \cap (F_1 \cup F_2)_\star = \emptyset$.  Thus $\pi(\alpha_{i_0+1}') \cap \pi(P_1' \cup P_2') = \emptyset$, since $P_1'$ and $P_2'$ are contained in $(F_1 \cup F_2)_\star$.

\bigskip
This completes the proof of all the properties required to prove that $\langle S_1',\ldots,S_k'; P_1',P_2' \rangle$ is a broken stairwell structure for $S' = S_1' \cup \cdots \cup S_k' \cup P_1' \cup P_2'$ of height $k$ with a pit at level $i_0 + 1$, or, in the case that $i_0 = k$, that $\langle S_1',\ldots,S_k' \rangle$ is a stairwell structure for $S' = S_1' \cup \cdots \cup S_k'$.

Finally, to obtain a connected simple fold, we observe that since $\setends{P_1} \neq \emptyset$ and $\gamma_{i_0} \neq \emptyset$ (by assumption), and since $G$ is connected and $\pi(P_2)$ has consistent complement relative to $\pi(\setends{P_1})$ and to $\pi(\gamma_{i_0})$, there is a component $K$ of $\pi(P_2) = \varphi(F_2)$ such that $K$ meets both $\pi(\setends{P_1}) = \partial \varphi(F_1)$ and $\pi(\gamma_{i_0}) = \partial \varphi(F_3)$.

By Proposition \ref{prop:conn fold}, $(\varphi {\upharpoonright}_{F_2})^{-1}(K)$ is contained in a component $C$ of $F$ such that $\varphi(C) = G$, and $F' = F_1' \cup F_2' \cup F_3'$, where $F_i' = F_i \cap C$ for each $i = 1,2,3$, is a connected simple fold.  For $i \neq i_0$, since $\varphi(C) = G$ and $S_i' = \varphi_\star^{-1}(S_i)$, we have $S_i' \cap C_\star \neq \emptyset$.  Also, all three of $S_{i_0}'$, $P_1'$, and $P_2'$ contain $\varphi_\star^{-1}(\gamma_{i_0}) \cap (\partial F_3)_\star$, and clearly $C$ meets $\partial F_3$ by Lemma \ref{lem:fold basic}\ref{enum:inter sep}.  Thus $S_{i_0}' \cap C_\star$, $P_1' \cap C_\star$, and $P_2' \cap C_\star$ are all non-empty as well.  Therefore, by the remarks following Definitions \ref{defn:stairwell} and \ref{defn:broken stairwell}, the (broken) stairwell structure on $S' \subset F_\star$ yields a (broken) stairwell structure on $S' \cap C_\star$.
\end{proof}

\section{Applications}
\label{sec:applications}

We are now in a position to state and prove our main technical theorem.

\begin{thm}
\label{thm:lift}
A compactum $X$ is hereditarily indecomposable if and only if for any map $f: X \to G$ to a graph $G$, for any set $M \subseteq G \times (0,1)$ which separates $G \times \{0\}$ from $G \times \{1\}$ in $G \times [0,1]$, for any open set $U \subseteq G \times [0,1]$ with $M \subseteq U$, and for any $\varepsilon > 0$, there exists a map $h: X \to U$ such that $\dsup(f, \pi_1 \circ h) < \varepsilon$ (where $\pi_1: G \times [0,1] \to G$ is the first coordinate projection).
\end{thm}

\begin{proof}
Suppose that $X$ is a hereditarily indecomposable compactum.  Let $f: X \to G$ be a map to a graph $G$, let $M \subset G \times (0,1)$ separate $G \times \{0\}$ from $G \times \{1\}$ in $G \times [0,1]$, and let $U$ be a neighborhood of $M$ in $G \times [0,1]$.  By treating the components of $G$ one at a time, and because the inverse image of any component under $f$ is a hereditarily indecomposable closed and open subset of $X$, we may assume without loss of generality that $G$ is connected.

By Theorem \ref{thm:get stairwell}, there is a set $S \subset U$ with a stairwell structure of odd height $k_0$.  We claim that there is a finite sequence $G = F^0,F^1,\ldots,F^n$ of connected graphs such that for each $i = 1,\ldots,n$, $F^i$ is a simple fold on $F^{i-1}$ with projection $\varphi_i: F^i \to F^{i-1}$, and $((\varphi_n)_\star \circ \cdots \circ (\varphi_1)_\star)^{-1}(S)$ contains a set $S'$ with a stairwell structure of height $1$.  We construct this sequence by induction as follows.  Let $F^0 = G$.

\setcounter{step}{0}
\begin{step}
Assume we have a set $S \subset (F^j)_\star$ with a stairwell structure of height $k$.  If $k = 1$, then we are done.  Otherwise, by Proposition \ref{prop:reduce height}, $S$ has a broken stairwell structure of height $k-2$ with a pit at level $1$.
\end{step}

\begin{step}
Assume that $S \subset (F^j)_\star$, and that $\langle S_1,\ldots,S_{k-2},P_1,P_2 \rangle$ is a broken stairwell structure on $S$ of height $k-2$ with a pit at level $i_0$.  As per the remarks following Proposition \ref{prop:reduce height}, if $\gamma_{i_0} = \emptyset$, then in fact $S$ has a stairwell structure of height $k-2$, and we may return to Step 1 with this stairwell structure.  Similarly, if $\setends{P_1} = \emptyset$, then in fact $S' = P_1 \subseteq S$ itself has a stairwell structure of height $1$, and we are done.

Suppose now that $\gamma_{i_0} \neq \emptyset$ and $\setends{P_1} \neq \emptyset$.  If $i_0 < k-2$, then by Proposition \ref{prop:unfold}, there is a simple fold $\varphi_j: F^{j+1} \to F^j$, where $F^{j+1}$ is a connected graph, and a set $S' \subseteq \varphi_j^{-1}(S)$ with a broken stairwell structure of height $k-2$ with a pit at level $i_0+1$, and we may repeat Step 2 for $S' \subset (F^{j+1})_\star$.  If $i_0 = k-2$, then by Proposition \ref{prop:unfold}, there is a simple fold $\varphi_j: F^{j+1} \to F^j$, where $F^{j+1}$ is a connected graph, and a set $S' \subseteq \varphi_j^{-1}(S)$ with a stairwell structure of height $k-2$, and we may repeat the entire process starting at Step 1 for $S' \subset (F^{j+1})_\star$.
\end{step}

In this way, after a sequence of at most $(k_0-1) + (k_0-3) + \cdots + 1$ simple folds, we obtain the desired sequence $G = F^0,F^1,\ldots,F^n$ and desired set $S' \subset (F^n)_\star$.  Clearly the first coordinate projection $\pi_1: F^n \times [0,1] \to F^n$ carries $S'$ one-to-one onto $F^n$, so there is an inverse $\theta: F^n \to S'$.

Let $g_0 = f$.  By Theorem \ref{thm:factor fold}, for each $i = 1,\ldots,n$, there is a map $g_i: X \to F^i$ such that $\dsup(\varphi_i \circ g_i,g_{i-1}) < \varepsilon_i$, where the numbers $\varepsilon_i > 0$ are chosen small enough so that if we let $g = \varphi_1 \circ \cdots \circ \varphi_n \circ g_n$, then $\dsup(f,g) < \varepsilon$.

Define $h: X \to G_\star$ by $h = (\varphi_1)_\star \circ \cdots \circ (\varphi_n)_\star \circ \theta \circ g_n$.  We further assume the numbers $\varepsilon_i$ are chosen small enough so that $h(X) \subset U$.  Then $\pi_1 \circ h = g$, hence $\dsup(f, \pi_1 \circ h) < \varepsilon$.

\bigskip
For the converse, assume $X$ is compact and that the right side of the ``if and only if'' statement holds.  Let $f: X \to [0,1]$ be a map, and let $\varphi: F \to [0,1]$ be a simple fold such that $F$ is an arc.  Consider a ``zig-zag'' set $S \subset [0,1] \times (0,1)$ which is the union of three straight sets $S_1,S_2,S_3 \subset [0,1] \times (0,1)$ such that $\pi_1(S_i) = \varphi(F_i)$ for each $i = 1,2,3$, $S_1 \cap S_2 = \setends{S_1}$, $S_2 \cap S_3 = \setends{S_3}$, and $S_1 \cap S_3 = \emptyset$.  Clearly $S$ separates $[0,1] \times \{0\}$ from $[0,1] \times \{1\}$ in the square $[0,1] \times [0,1]$.  Note also that there is a homeomorphism $\rho: S \to F$ such that $\varphi \circ \rho = \pi_1$ on $S$.

Fix $\varepsilon > 0$, and let $U$ be a small neighborhood of $S$ in $[0,1] \times [0,1]$ for which there is a $\frac{\varepsilon}{2}$-retraction $r: U \to S$ -- in particular, so that on $U$ we have $\dsup(\pi_1 \circ r, \pi_1) < \frac{\varepsilon}{2}$.  By hypothesis, there is a map $h: X \to U$ such that $\dsup(f, \pi_1 \circ h) < \frac{\varepsilon}{2}$.  Let $g = \rho \circ r \circ h: X \to F$.  Observe that $\varphi \circ g = \pi_1 \circ r \circ h$.  Then we have $\dsup(\varphi \circ g, \pi_1 \circ h) = \dsup(\pi_1 \circ r \circ h, \pi_1 \circ h) < \frac{\varepsilon}{2}$ and $\dsup(f, \pi_1 \circ h) < \frac{\varepsilon}{2}$, hence $\dsup(f, \varphi \circ g) < \varepsilon$.

Therefore, by Theorem \ref{thm:factor fold}, $X$ is hereditarily indecomposable.
\end{proof}

We now recall and prove Theorem \ref{thm:hered indec span zero}, from which the classification of homogeneous plane continua (and compacta) follows as detailed in the Introduction above.

\setcounter{thm}{0}
\begin{thm}
A continuum $X$ is homeomorphic to the pseudo-arc if and only if $X$ is hereditarily indecomposable and has span zero.
\end{thm}

\begin{proof}
The pseudo-arc is hereditarily indecomposable and arc-like, and all arc-like continua have span zero \cite{lelek64}, hence the pseudo-arc has span zero.

For the converse, let $X$ be a hereditarily indecomposable continuum in the Hilbert cube $[0,1]^{\mathbb{N}}$ with span zero, and fix $\varepsilon > 0$.  We will show there is an $\varepsilon$-map from $X$ to an arc.

By Theorem \ref{thm:span separator}, there exists $\delta > 0$ small enough so that if $G \subset [0,1]^{\mathbb{N}}$ is a graph and $I \subset [0,1]^{\mathbb{N}}$ is an arc with endpoints $p$ and $q$, such that the Hausdorff distance from $X$ to each of $G$ and $I$ is less than $\delta$, then the set $M = \{(x,y) \in G \times (I \smallsetminus \{p,q\}): d(x,y) < \frac{\varepsilon}{6}\}$ separates $G \times \{p\}$ from $G \times \{q\}$ in $G \times I$.  We may assume $\delta \leq \frac{\varepsilon}{6}$.

Let $I \subset [0,1]^{\mathbb{N}}$ be an arc with endpoints $p$ and $q$ such that $d_H(X,I) < \delta$.  Since $X$ has span zero, by \cite{lelek79} we have that $X$ is tree-like.  Therefore, there exists a tree $T \subset [0,1]^{\mathbb{N}}$ and a map $f: X \to T$ such that $\dsup(f,\id_X) < \delta$.  It follows that $d_H(X,T) < \delta$.  Hence, by choice of $\delta$, the set $M = \{(x,y) \in T \times (I \smallsetminus \{p,q\}): d(x,y) < \frac{\varepsilon}{6}\}$ separates $T \times \{p\}$ from $T \times \{q\}$ in $T \times I$.

Let $\pi_1: T \times I \to T$ and $\pi_2: T \times I \to I$ denote the first and second coordinate projections, respectively.  Since $M$ is open, by Theorem \ref{thm:lift} there is a map $h: X \to M$ such that $\dsup(f, \pi_1 \circ h) < \frac{\varepsilon}{6}$.

We claim that $\pi_2 \circ h: X \to I$ is such that $\dsup(\pi_2 \circ h, \id_X) < \frac{\varepsilon}{2}$, which means that $\pi_2 \circ h$ is an $\varepsilon$-map.  Indeed, given $x \in X$, we have
\begin{align*}
d(x, \pi_2 \circ h(x)) &\leq d(x,f(x)) + d(f(x), \pi_1 \circ h(x)) + d(\pi_1 \circ h(x), \pi_2 \circ h(x)) \\
&< \delta + \frac{\varepsilon}{6} + \frac{\varepsilon}{6} \qquad \textrm{since $\dsup(f,\id_X) < \delta$, $\dsup(f, \pi_1 \circ h) < \frac{\varepsilon}{6}$, and $h(x) \in M$} \\
&\leq \frac{\varepsilon}{2} \qquad \textrm{since $\delta \leq \frac{\varepsilon}{6}$.}
\end{align*}

Therefore $X$ is arc-like.  Because $X$ is hereditarily indecomposable and arc-like, it is homeomorphic to the pseudo-arc \cite{bing51}.
\end{proof}
\setcounter{thm}{20}

\section{Discussion and questions}
\label{sec:discussion}

A closely related classification problem of significant interest is: What are all the homogeneous hereditarily indecomposable continua?  This question was asked by Jones in \cite{jones55s}.  It is known, by results of Prajs and Krupski \cite{KP90} and of Rogers \cite{rogers82}, that a homogeneous continuum is hereditarily indecomposable if and only if it is tree-like.  Thus far, the pseudo-arc is the only known example of a non-degenerate homogeneous tree-like continuum.

\begin{question}
If $X$ is a homogeneous tree-like (equivalently, hereditarily indecomposable) continuum, must $X$ be homeomorphic to the pseudo-arc?
\end{question}

By the results of this paper, if there is another such continuum, it would necessarily be non-planar.  An affirmative answer to this question would follow if one could prove that every homogeneous tree-like continuum has span zero.  The question of whether every homogeneous tree-like continuum has span zero was raised by Ingram in \cite[Problem 93]{Houston95}.

\bigskip
Theorem \ref{thm:lift} can also be applied to the study of hereditarily equivalent spaces.  A continuum $X$ is \emph{hereditarily equivalent} if $X$ is homeomorphic to each of its non-degenerate subcontinua.  In a forthcoming paper \cite{HO2015}, the authors use Theorem \ref{thm:lift} to show that the only non-degenerate hereditarily equivalent plane continua are the arc and the pseudo-arc.

\bigskip
Recall from the comments immediately preceding Theorem \ref{thm:span separator} that a continuum $X$ has \emph{surjective semispan zero} \cite{lelek77} if every subcontinuum $Z \subseteq X \times X$ with $\pi_2(Z) = X$ meets the diagonal $\Delta X = \{(x,x): x \in X\}$.  It is proved in \cite{OT84} that any continuum with surjective semispan zero is tree-like.  Our proof of Theorem \ref{thm:hered indec span zero} in fact establishes the following slightly stronger characterization of the pseudo-arc:

\newtheorem*{thm semispan}{Theorem \ref{thm:hered indec span zero}$'$}
\begin{thm semispan}
A continuum $X$ is homeomorphic to the pseudo-arc if and only if $X$ is hereditarily indecomposable and has surjective semispan zero.
\end{thm semispan}

It is clear that every continuum with span zero has surjective semispan zero, but it is not known whether these two properties are equivalent.

\begin{question}[cf.\ {\cite[Problem 59]{Houston95}}]
Does every continuum with surjective semispan zero have span zero?
\end{question}

\bigskip
Besides the property of span zero, another property related to arc-likeness is weak chainability: a continuum is \emph{weakly chainable} if it is the continuous image of an arc-like continuum.  This concept was introduced by Lelek \cite{lelek62} who used an equivalent formulation involving weak chain covers.  Lelek \cite{lelek62} and Fearnley \cite{fearnley64} independently proved the equivalence of these two notions, and observed that every arc-like continuum is the continuous image of the pseudo-arc, which means that a continuum is weakly chainable if and only if it is the continuous image of the pseudo-arc.

It is known that all arc-like continua have span zero \cite{lelek64}, and all span zero continua are weakly chainable \cite{OT84}.  Therefore, if the answer to the following question is affirmative, it would yield a still stronger characterization of the pseudo-arc than our Theorem \ref{thm:hered indec span zero}.

\begin{question}
If $X$ is a hereditarily indecomposable and weakly chainable continuum, must $X$ be homeomorphic to the pseudo-arc?
\end{question}

It is known (see e.g.\ \cite{LR74} and \cite{mclean72}) that a hereditarily indecomposable and weakly chainable continuum must be tree-like.

\bigskip
It is possible to formulate a version of Theorem \ref{thm:lift} without any mention of separators in the product of a graph with an arc, which more directly generalizes Theorem \ref{thm:factor fold}.  To this end, we give a generalization of the notion of a simple fold (Definition \ref{defn:simple fold}), which is inspired by our definition of a stairwell structure (Definition \ref{defn:stairwell}).

\begin{defn}
\label{defn:folding map}
A \emph{folding map} on $G$ is a graph $F = F_1 \cup \cdots \cup F_k$ and a function $\varphi: F \to G$, called the \emph{projection}, which satisfy the following properties.  Let $G_i = \varphi(F_i)$ for $i = 1,\ldots,k$.
\begin{enumerate}[label=\textbf{(FM\arabic{*})}]
\item $k$ is odd, and $G_1,\ldots,G_k$ are regular subsets of $G$;
\item For each $i = 1,\ldots,k$, $\partial G_i = A_i \cup B_i$, where $A_i$ and $B_i$ are disjoint finite sets, $A_1 = B_k = \emptyset$, and $B_i = A_{i+1}$ for each $i = 1,\ldots,k-1$;
\item For each $i = 1,\ldots,k-1$ there is a neighborhood $V$ of $B_i = A_{i+1}$ such that $G_i \cap V = G_{i+1} \cap V$;
\item For each $i = 1,\ldots,k$, $G_i$ has consistent complement relative to $A_i$ and to $B_i$;
\item $\varphi {\upharpoonright}_{F_i}$ is a homeomorphism $F_i \to G_i$ for each $i = 1,\ldots,k$; and
\item $\varphi(F_i \cap F_{i+1}) = B_i = A_{i+1}$ for each $i = 1,\ldots,k-1$, and $F_i \cap F_j = \emptyset$ whenever $|i - j| > 1$.
\end{enumerate}
\end{defn}

It is straightforward to see that given a folding map $\varphi: F \to G$ to a connected graph $G$, one can construct a set $S \subset G \times (0,1)$ with a stairwell structure corresponding to $\varphi$ as in the proof of Theorem \ref{thm:lift}.  In this way, one can prove the following result.

\begin{thm}
\label{thm:factor folding map}
A compactum $X$ is hereditarily indecomposable if and only if for any map $f: X \to G$ to a connected graph $G$, for any folding map $\varphi: F \to G$, and for any $\varepsilon > 0$, there exists a map $g: X \to F$ such that $\dsup(f, \varphi \circ g) < \varepsilon$.
\end{thm}

Observe that the linear ordering of the sets $F_1,\ldots,F_k$, where each of these sets meets only its immediate successor and predecessor, is an essential feature which causes the correspondance between folding maps and sets in $G \times (0,1)$ with stairwell structures (for connected graphs $G$).  However, inspired by the notion of a broken stairwell structure, one could formulate a more general concept of a folding map, in which the adjacency relation on the sets $F_1,\ldots,F_k$ (here we say $F_i$ and $F_j$ are \emph{adjacent} if $F_i \cap F_j \neq \emptyset$) is a tree (or more generally any graph), instead of an arc (linear order).

\begin{question}
Can one prove a version of Theorem \ref{thm:factor fold} (and Theorem \ref{thm:factor folding map}) which pertains to a notion of folding maps $\varphi: F \to G$ for which the subgraphs of $F$ on which $\varphi$ is one-to-one are allowed to have an adjacency relation which is a tree?  More generally, under what conditions on this adjacency relation does there exist, for any map $f: X \to G$ from a hereditarily indecomposable compactum $X$ and any $\varepsilon > 0$, a map $g: X \to F$ such that $\dsup(f, \varphi \circ g) < \varepsilon$?
\end{question}

\bibliographystyle{amsplain}
\bibliography{Separators}

\end{document}